\documentclass[11pt]{amsart}

\usepackage{macros}
\usepackage{mathtools}

\def\co{\colon\thinspace}

\DeclareMathOperator{\id}{id}
\DeclareMathOperator{\dd}{d}

\def\op{\mathrm}
\def\lalgd{\mathsf{LieAlgd}}
\def\lalgdbf{\mathsf{LieAlgd}_{\rm bf}}
\def\dglalgd{\mathsf{dgLieAlgd}}
\def\dglalgdbf{\mathsf{dgLieAlgd}_{\rm bf}}
\def\linfcat{\L8$-$\mathsf{space}}
\def\linfcatbf{\L8$-$\mathsf{space}_{\rm bf}}
\def\dstk{\mathsf{dSt}}
\def\fib{\op{fib}}

\def\shom{{\cH}om}
\def\VB{\mathop{\rm VB}}
\def\Mod{\mathop{\rm Mod}}
\def\colim{\mathop{\rm colim}}
\def\repinf{\RR {\rm ep}^\infty}

\def\calgd{\mathsf{C}^*}

\def\xto{\xrightarrow}

\def\calsym{\widehat{{\cS} {\rm ym}}}

\title{Lie algebroids as $\L8$ spaces}

\author{Ryan Grady}
\address{Department of Mathematical Sciences\\Montana State University\\Bozeman, MT 59717}
\email{ryan.grady1@montana.edu}
\thanks{The first author was partially supported by the National Science Foundation under Award DMS-1309118.}

\author{Owen Gwilliam}
\address{Max Planck Institut f\"ur Mathematik\\Vivatsgasse 7\\53111 Bonn\\Germany}
\email{gwilliam@mpim-bonn.mpg.de}
\thanks{The second author was partially supported as a postdoctoral fellow by the National Science Foundation under Award DMS-1204826.}

\subjclass[2010]{58H15, 53D17, 14D23}

\keywords{dg Lie algebroid; derived geometry; shifted symplectic structure}

\begin{document}

\begin{abstract}
In this paper, we relate Lie algebroids to Costello's version of derived geometry. For instance, we show that each Lie algebroid---and the natural generalization to dg Lie algebroids---provides an (essentially unique) $L_\infty$ space. More precisely, we construct a faithful functor from the category of Lie algebroids to the category of $L_\infty$ spaces. Then we show that for each Lie algebroid $L$, there is a fully faithful functor from the category of representations up to homotopy of $L$ to the category of vector bundles over the associated $L_\infty$ space. Indeed, this functor sends the adjoint complex of $L$ to the tangent bundle of the $L_\infty$ space. Finally, we show that a shifted-symplectic structure on a dg Lie algebroid  produces a shifted symplectic structure on the associated $L_\infty$ space.
\end{abstract}

\maketitle

\tableofcontents

\section{Introduction}

Lie algebroids appear throughout geometry and physics, and they provide a fertile transfer of ideas and intuition between geometry and Lie theory. (The literature is vast. See, for instance, \cite{Mackenzie,dSW,ELW,Fernandes,CRVDB,CF}, among many other papers.)  Recently the language of derived geometry has provided another perspective on the relationship between Lie theory and geometry, emphasizing the idea that a deformation problem is describable by a dg Lie algebra (e.g., \cite{LurieDAGX,Hinich}). (The formal story has a rich history built on ideas of Schlessinger, Stasheff, Quillen, Illusie, Deligne, Drinfeld, Kontsevich, and others, see \cite{Manetti} and the references contained therein.) These two approaches are compatible (see, e.g., \cite{Hennion,CCT,GR}), and by combining them, it becomes clearer both how to systematically provide Lie algebroid versions of constructions from Lie theory and also how to interpret such constructions in derived geometric terms. 

In this paper we will use instead an approach to derived geometry initiated by Costello \cite{CosWG,GGLoop} for two related reasons. First, Costello's notion of $\L8$ space interpolates smoothly between the functorial approach to derived geometry and the language of dg manifolds, or $Q$-manifolds, common in mathematical physics.  Hence it is convenient for drawing from the rich literature in higher differential geometry, Second, as we describe in subsection \ref{ss:apps}, Costello's formalism is compatible with his machinery for renormalization \cite{Cos1} and hence makes it possible to rigorously develop interesting perturbative quantum field theories.

\subsection{What we prove}

Our main results in this paper all amount to showing that a well-established notion in Lie algebroids maps to a parallel notion in Costello's version of derived geometry. For instance, we show that each Lie algebroid $L$---and the natural generalization to dg Lie algebroids---provides an (essentially unique) $\L8$ space $\op{enh}(L)$. More precisely, we construct a faithful functor from the category of Lie algebroids to the category of $\L8$ spaces. Then we show that for each Lie algebroid $L$, there is a fully faithful functor from the category of representations up to homotopy of $L$ to the category of vector bundles over $\op{enh}(L)$. Indeed, this functor sends the adjoint complex of $L$ to the tangent bundle of $\op{enh}(L)$. Finally, we show that a shifted-symplectic structure on a dg Lie algebroid $L$ produces a shifted-symplectic structure on~$\op{enh}(L)$.

\begin{remark}
These results are not tautological. The definitions of vector bundle and symplectic structure for $\L8$ spaces were written before we knew about representations up to homotopy or symplectic Lie $n$-algebroids. Indeed, our results show a fortuitous alignment between these two approaches to higher structures in differential geometry;  the simplicity of the relationship between Lie algebroids and $\L8$ spaces surprised us. We hope this pattern continues. (We wonder, in particular, about $\L8$ space analogs of, e.g., the work of Calaque, C{\u{a}}ld{\u{a}}raru, and Tu \cite{CCT} on Lie algebroids for derived intersections.)
\end{remark}

\subsection{Applications to physical mathematics}
\label{ss:apps}

Costello introduced $\L8$ spaces and his framework for derived geometry to facilitate the expression of classical field theories, particularly nonlinear $\sigma$-models, in a manner amenable to quantization via Feynman diagrams and renormalization. These notions appeared in his quantization of the curved $\beta\gamma$ system \cite{CosWG}. Since this work, his methods have been applied to several more examples:
\begin{enumerate}
\item[(i)] one-dimensional topological $\sigma$-model into a cotangent bundle \cite{GGCS};
\item[(ii)] one-dimensional topological $\sigma$-model into a symplectic manifold, recovering Fedosov quantization from the BV formalism \cite{GLL};
\item[(iii)] the topological B-model and the Landau-Ginzburg model \cite{LiLi};
\item[(iv)] the two-dimensional nonlinear $\sigma$-model \cite{GW, Nguyen}.
\end{enumerate}
The results in those papers, though, extend to a much larger class of target spaces: typically any $\L8$ space satisfying some analog of the geometric structure required when the target is an ordinary manifold (e.g., a symplectic form). Written in this style, these nonlinear $\sigma$-models are simply versions of Chern-Simons, holomorphic Chern-Simons, or BF theories with a sophisticated version of a Lie algebra as the ``gauge group'' (rather, the gauge algebra).

Our work thus allows us to formulate nonlinear $\sigma$-models with Lie algebroids as the target spaces. Another intriguing direction is to formulate the Lie algebroid Yang-Mills theories of Strobl and collaborators \cite{Strobl, KStrobl, MStrobl} in terms of $\L8$ spaces, in order to consider their perturbative quantizations.

\subsection{Notations and conventions}

We work throughout in characteristic zero.
We work cohomologically, so the differential in any complex increases degree by one.

For $A$ a cochain complex, $A^\sharp$ denotes the underlying graded vector space. If $A$ is a cochain complex whose degree $k$ space is $A^k$, then $A[1]$ is the cochain complex where $A[1]^k = A^{k+1}$. We use $A^\vee$ to denote the graded dual.

For $X$ a smooth manifold, we use $T_X$ to denote its tangent bundle as a vector bundle and $TX$ to denote the total space of that vector bundle. We use $T^\vee_X$ to denote the cotangent bundle.

If $f \co X \to Y$ is a map of smooth manifolds and $V$ a vector bundle, then we use $f^{-1} V$ to denote the pullback vector bundle. Similarly, for $\cF$ a sheaf on $Y$, we use $f^{-1}\cF$ to denote the pullback sheaf, simply as a sheaf of sets or vector spaces. We reserve the notation $f^\ast \cV$ for the case where $\cV$ is a sheaf of dg $\Omega^\ast_Y$-modules, and $f^\ast \cV$ denotes the sheaf of dg $\Omega^\ast_X$-modules obtained from $f^{-1} \cV$ by extending scalars.

\subsection{Acknowledgements}

We have many people to thank for their encouragement and help in pursuing these questions. Chenchang Zhu emphasized that such results as we prove here would be worthwhile.  Ping Xu and Mathieu Sti\'enon have discussed with us various approaches to higher differential geometry with persistence and enthusiasm. Damien Calaque has had a strong influence on us, particularly his artful interweaving of Lie algebroids, jets, and derived geometry in his work with Van den 
Bergh and others. Both the work of Arias Abad-Crainic on representations up to homotopy \cite{AbadCrainic} and the work of Crainic-Moerdijk on the deformation complex \cite{CrainicMoerdijk} present a clean and powerful approach to Lie algebroids which made our work possible (and much more conceptual). We owe a great debt to Kevin Costello for his inspiring discussions and regular feedback. Finally, we thank the anonymous referee who identified a gap in an important argument and made several quite helpful suggestions that improved the paper.

The authors gratefully acknowledge support from the Hausdorff Research Institute for Mathematics, especially the trimester program ``Homotopy theory, manifolds, and field theories'' at which some of the research for this paper was performed. Similarly, we would like to thank both the Max Planck Institut f\"ur Mathematik and the Perimeter Institute for Theoretical Physics for their hospitality and working conditions during the writing of parts of the current paper. Research at Perimeter Institute is supported by the Government of Canada through Industry Canada and by the Province of Ontario through the Ministry of Economic Development and Innovation.

\section{Recollections on $\L8$ spaces}

We will give a brief overview of the definitions and constructions with $\L8$ spaces that are relevant to our work here. For more detail and exposition, see~\cite{GGLoop}.

\subsection{Curved $\L8$ algebras}

We begin by describing the relevant definitions from algebra.
Recall that for a graded commutative algebra $A^\sharp$ and a graded $A^\sharp$-module $V$, the graded $A^\sharp$-module 
\[
\Sym_{A^\sharp}(V)= \bigoplus_{n \geq 0} \Sym^n_{A^\sharp}(V)
\]
admits a natural cocommutative coalgebra structure in which
\[
\Delta(v_1 \cdots v_n) = \sum \pm v_{\sigma(1)} \cdots v_{\sigma(p)} \otimes v_{\sigma(p+1)} \cdots v_{\sigma(n)}
\]
where the sum is over all $(p,q)$-shuffles $\sigma$ with $p+q = n$ and the sign is via the Koszul rule.

Recall that a map of coalgebras is determined by its image into the cogenerators.  In particular, if $\phi \co \Sym (V) \to \Sym (W)$ is a map of coalgebras, then $\phi$ is determined by a collection of maps 
\[
\{\phi_n \co \Sym^n (V) \to W \oplus A^\sharp,  \text{ for } n \ge 0\},
\]
where each component $\phi_n$ denotes the restriction of $\phi$ to the summand $\Sym^n (V)$ followed by projection onto the cogenerator $W \oplus A^\sharp = \Sym^1 W \oplus \Sym^0 W$. (If the map $\phi$ respects the usual coaugmentations by $A^\sharp$, then $W$ is a cogenerator and it thus suffices to consider the projection onto just~$W$.)

\begin{definition}\label{curvedl8}
Let $A$ be a commutative dg algebra with a nilpotent dg ideal $I$ (i.e., $I^n = 0$ for some $n$). 
Let $A^\sharp$ denote the underlying graded commutative algebra. A {\it curved $\L8$ algebra $\fg$ over $A$} consists~of
\begin{enumerate}
\item[(1)] a locally free, $\ZZ$-graded $A^\sharp$-module $V_\fg$ and
\item[(2)] a linear map of cohomological degree 1
\[
d\co \Sym_{A^\sharp} (V_\fg[1]) \to  \Sym_{A^\sharp} (V_\fg[1])
\]
\end{enumerate}
such that
\begin{enumerate}
\item[(i)] $(\Sym_{A^\sharp} (V_\fg[1]),d)$ is a cocommutative dg coalgebra over $A$ with the standard coproduct, and
\item[(ii)] $d(\Sym^0(V_\fg[1])) \subset I \cdot V_\fg[1]$.
\end{enumerate}
We use $C_\ast(\fg)$ to denote this cocommutative dg coalgebra and call it the {\it Chevalley-Eilenberg homology complex} of this curved $\L8$ algebra~$\fg$.
\end{definition}

These conditions amount to requiring that $d$ be a square-zero coderivation and that modulo $I$, the coderivation $d$ vanishes on the constants. In short, base-changing along the algebra map $A \to A/I$, we obtain an un-curved $\L8$ algebra.

As usual in the $\L8$ setting, we use the notation and terminology ``Chevalley-Eilenberg'' since these constructions extend the usual notions of Lie algebra homology. (Note, however, that we work cohomologically, so that our differential still increases degree. Thus, for us, the homology complex of any ordinary Lie algebra $\fg$ is concentrated in {\it nonpositive} degrees.)

There is also a natural Chevalley-Eilenberg {\it cohomology} complex $\widehat{C}^*(\fg)$, defined as follows.
For $V$ a graded $A^\sharp$-module, its {\em completed symmetric algebra} is the graded $A^\sharp$-module
\[
\csym_{A^\sharp} (V) = \prod_{n \geq 0} \Sym^n_{A^\sharp}(V)
\]
equipped with the filtration $F^k \csym_{A^\sharp} (V)  = \Sym^{\geq k}_{A^\sharp}(V)$ and the usual commutative product, which is filtration-preserving. 
Then $\widehat{C}^*(\fg)$ is $(\csym_{A^\sharp}(V_\fg^\vee[-1]), d_\fg)$ with $d_\fg$ the differential dual to $d$ on $C_*(V)$. In particular, $d$ is a derivation. 

Note that powers of the nilpotent ideal $I$ provides another natural filtration on $C_*(V)$ and on $\widehat{C}^*(V)$: for example, $F^k_I C_*(V) = I^k \cdot C_*(V)$. We write $\Gr C_*(V)$ for the associated graded cocommutative dg coalgebra of this filtration.

Similarly, powers of the nilpotent ideal $I$ equips $\fg$ and the underlying vector space $V_\fg$ with filtrations.  The associated graded $\Gr \fg$ is a $\L8$ algebra over $(\Gr A, 0)$, where $\Gr A$ is the associated graded to the $I$-filtration. Note that $\Gr V_\fg [1]$ has no curving, and, in particular, $\Gr V_\fg [1]$ is thus a cochain complex.

\begin{definition}
A {\it map of curved $\L8$ algebras} $\phi \co \fg \to \fh$ is a map of cocommutative dg coalgebras $\phi \co C_*(\fg) \to C_*(\fh)$ respecting the $I$-filtration. A map is a {\it weak equivalence} if the map $\Gr ( \phi_1 ) \co \Gr V_\fg [1] \to \Gr V_\fh[1]$ on the associated graded cochain complexes is a quasi-isomorphism.
\end{definition}

Recall our convention, stated just before Definition \ref{curvedl8}, for the ``components'' of a coalgebra map given by projection onto a cogenerator. For a non-curved $\L8$ algebra, such as $\Gr \fg$ or $\Gr \fh$, the Chevalley-Eilenberg homology complex is naturally coaugmented and hence $\Gr V_\fh[1]$ is a cogenerator.

\begin{remark}
\mbox{}
\begin{enumerate}
\item Note that this notion of weak equivalence is stronger than requiring a quasi-isomorphism between Chevalley-Eilenberg homology complexes, or even a filtered quasi-isomorphism between homology complexes (i.e., a quasi-isomorphism on the associated gradeds). Indeed, note that we have an isomorphism $\Gr C_\ast \fg \cong C_\ast \Gr \fg$, so that a weak equivalence induces a filtered quasi-isomorphism.
\item In \cite{Hinich}, Hinich equips the category of conilpotent cocommutative coalgebras over a field of characteristic zero with a (non-obvious) model structure such that the Chevalley-Eilenberg complex $C_\ast (-)$ is a right Quillen functor from a model category of differential graded Lie algebras and this Quillen adjunction is a Quillen equivalence. The definition of weak equivalence of curved $\L8$ algebras is a natural extension of Hinich's notion, as it agrees with his on the associated graded non-curved $\L8$ algebras.
\end{enumerate}
\end{remark}

\subsection{$\L8$ spaces}\label{sect:L8spaces}

We now describe a version of ``families of curved $\L8$ algebras parametrized by a smooth manifold."

\begin{definition}
Let $X$ be a smooth manifold. An {\it $\L8$ space} is a pair $(X, \fg)$, where $\fg$ is the sheaf of smooth sections of a finite-rank, $\ZZ$-graded vector bundle $\pi\co V_\fg \to X$ equipped with the structure of a curved $\L8$ algebra structure over the commutative dg algebra $\Omega^\ast_X$ with nilpotent ideal~$\sI=\Omega^{\geq 1}_X$.
\end{definition}

\begin{remark}
In \cite{GGLoop}, we allowed the vector bundle $V_\fg$ to be a topological vector bundle in order to include a class of examples related to nonlinear $\sigma$-models. The fibers were Fr\'echet vector spaces. (It might be better to tame such infinite-rank vector bundles by viewing them as bornological or as sheaves on a site of manifolds.) As we only need finite-rank vector bundles here, we will restrict to that case.
\end{remark}

As we explain in the next subsection, the best way to think of an $\L8$ space is via its functor of points. In other words, an $\L8$ space presents a more intrinsic geometric object, much as one can present a smooth manifold by an atlas or a nice topological space by a cell complex. But it is clear, even just from the definition, that there are many examples. For example, every $\L8$ algebra provides an $\L8$ space over a point. Less obvious examples arise from smooth and complex manifolds, as detailed in \cite{CosWG} and \cite{GGLoop}. Our main result here shows how to construct an $\L8$ space from a Lie algebroid.

Let $(X,\fg)$ be an $\L8$ space. Observe that given a smooth map $f\co Y \to X$, we obtain a curved $\L8$ algebra over $\Omega^*_Y$~by
\[
f^* \fg := f^{-1}\fg \otimes_{f^{-1} \Omega^*_X} \Omega^*_Y,
\]
where $f^{-1} \fg$ denotes the sheaf of smooth sections of the pullback vector bundle~$f^{-1}~V_\fg$.

\begin{definition}
Let $(X, \fg)$ and $(Y, \fh)$ be $\L8$ spaces. A {\it map of $\L8$ spaces} $\Psi \co (X, \fg) \to (Y, \fh)$ is a pair $(f,\psi)$, where $f \co X \to Y$ is a smooth map and $\psi \co \fg \to f^\ast \fh$ is a map of curved $\L8$ algebras over $\Omega^\ast_X$. We say a map is {\it base-fixing} if $f$ is the identity.
\end{definition}

There are thus two categories of interest to us. Let $\linfcat$ denote the category of all $\L8$-spaces and all maps of $\L8$-spaces. For each manifold $X$, there is also the category $\linfcatbf(X)$ whose objects are $\L8$-spaces with underlying manifold $X$ and whose morphisms are the base-fixing maps thereof. (These are simply 1-categories, not $(\infty,1)$-categories.) 

In parallel to these 1-categories, there are two natural $(\infty,1)$-categories of $\L8$-spaces. We will present them as categories with weak equivalences, but first we will explain how $\L8$ spaces define derived stacks, as that functorial context determines the correct notion of weak equivalence.

\subsection{A functorial view on derived geometry}\label{sect:dstack}

An $\L8$ space has an associated ``functor of points'' and hence can be understood as presenting a kind of space in the same way that a commutative algebra presents a scheme. In the  formalism developed in \cite{GGLoop}, which we now briefly discuss, we make this assertion precise as follows. 

There is a site $\dgMan$ (in fact, an $\infty$-site) of {\it nil dg manifolds}, in which an object $\cM$ is a smooth manifold $M$ equipped with a sheaf $\sO_\cM$ of commutative dg algebras over $\Omega^*_M$ that has a nilpotent dg ideal $\sI_\cM$ such that $\sO/\sI \cong \cinf_M$. For the full definition, including the definition of cover, see \cite{GGLoop}. The details are not relevant for the constructions in this paper.

\begin{definition}\label{def:derivedstack}
A {\it derived stack} is a functor $\bX\co \dgMan^{op} \to s\!\Sets$ satisfying
\begin{enumerate}
\item[(1)] $\bX$ sends weak equivalences of nil dg manifolds to weak equivalences of simplicial sets;
\item[(2)] $\bX$ satisfies \v{C}ech descent, i.e., if for every nil dg manifold $\cM$ and every cover $\fV$ of $\cM$, we have
\[
\cF(\cM) \xrightarrow{\simeq} \holim_{\check{C}\fV}\,\cF,
\]
where $\check{C}\fV_\bullet$ denotes the \v{C}ech nerve of the cover (namely the simplicial diagram with $n$-simplices $\check{C}\fV_n := \fV \times_\cM \cdots \times_\cM \fV$).
\end{enumerate} 
\end{definition}

We now explain how every $\L8$ space defines such a functor.  

\begin{definition}
For $(X,\fg)$ an $\L8$ space, its {\it functor of points} $\bB\fg\co \dgMan^{op} \to s\!\Sets$
sends the nil dg manifold $\cM$ to the simplicial set $\bB\fg(\cM)$ in which an $n$-simplex is a pair $(f,\alpha)$: a smooth map $f\co M \to X$ and a solution $\alpha$ to the Maurer-Cartan equation in sections over $M$ of the $\L8$ algebra~$f^*\fg\ot_{\Omega^*_M}\sI_\cM\ot_\RR \Omega^*(\triangle^n)$.
\end{definition}

\begin{remark}
If we view $\widehat{C}^* \fg$ as the structure sheaf of the $\L8$ space, then a 0-simplex of $\bB\fg(\cM)$ is a map of the underlying manifolds $f\co M \to X$ and a map of commutative dg algebras $f^\ast \widehat{C}^* \fg \to \sO_\cM$. In other words, it is a map of dg ringed spaces. In practice, it is fruitful to think of an $\L8$ space as a dg manifold, described Koszul-dually as an $\L8$ algebra. But, as usual, the tricky aspect is to keep track of weak equivalences between dg manifolds, which is why the functorial approach is so helpful: one focuses on the output of a construction, not its inner workings.
\end{remark}

A central result of \cite{GGLoop} is then the following.

\begin{theorem}[Theorem 4.8 \cite{GGLoop}]\label{thm:L8IsDerived}
The functor $\bB \fg$ associated to an $\L8$ space $(X,\fg)$ is a derived stack.
\end{theorem}

This result thus gives a perspective on what an $\L8$ space means: it is a computationally-convenient presentation of a derived stack $\bB\fg$, which is a more invariant or intrinsic notion.

Derived stacks form a category $\dstk$ with weak equivalences, where a natural transformation $F\co \bX \to \bY$ between derived stacks is a weak equivalence if $F(\cM)\co \bX(\cM) \to \bY(\cM)$ is a weak homotopy equivalence for every nil dg manifold $\cM$. Because we only care about $\L8$ spaces in terms of their derived stacks, we want a notion of weak equivalence on $\L8$ spaces that matches with that on derived stacks. 

\begin{definition}
A map of $\L8$ spaces $\Psi \co (X,\fg) \to (Y,\fh)$ is {\it weak equivalence} if the map of underlying manifolds $f$ is a diffeomorphism and the map of curved $\L8$ algebras $\psi \co \fg \to f^\ast \fh$ is a weak equivalence.
\end{definition}

Our definition of weak equivalence between $\L8$ spaces is motivated by the following property.

\begin{prop}
The functor $\bB\co \linfcat \to \dstk$ sending $(X,\fg)$ to $\bB\fg$ is a functor between categories with weak equivalences, i.e., it preserves weak equivalences.  Moreover, this functor detects weak equivalences. 
\end{prop}

In short, a map of $\L8$ spaces $\Psi \co (X, \fg) \to (Y, \fh)$ is a weak equivalence if and only if the induced map $\bB \Psi \co \bB \fg \to \bB \fh$ is a weak equivalence of derived stacks.

\begin{proof}
Let us prove the first claim. Assume that $\Psi \co (X, \fg) \to (Y, \fh)$ is a weak equivalence. We will argue by artinian induction, i.e., by working up the natural tower of commutative dg algebras 
\[
\cO_\cM \xto{q_{k-1}} \cO_\cM/\cI_\cM^{k-1} \xto{q_{k-2}} \cdots \xto{q_2} \cO_\cM/\cI_\cM^{2} \xto{q_1} \cO_\cM/\cI_\cM \cong \cinf_M
\]
for any nil dg manifold $\cM = (M, \cO_\cM)$ whose characterizing nilpotent dg ideal $\cI_\cM$ is nilpotent of order $k$. Let $M_{sm}$ denote the nil dg manifold $(M,\cinf_M)$ associated to the smooth manifold $M$. The simplicial set $\bB\fg(M_{sm})$ is the discrete simplicial set given by the set of smooth maps $\op{Maps}(M,X)$. As the underlying smooth map $f \co X \to Y$ for $\Psi$ is a diffeomorphism, the map of sets $f \circ -\co \op{Maps}(M,X) \to \op{Maps}(M,Y)$ is an isomorphism. Without loss of generality, we can now assume that $X=Y$ and the underlying map $f$ is the identity.  

Fix a smooth map $\eta \co M \to X$; it is sufficient to consider the component over this fixed map. Now, by Lemma C.3 of \cite{GGLoop}, we know that the~map
\[
\bB \fg(q_{d})\co \bB \fg(M, \cO_\cM/\cI_\cM^{d+1}) \to \bB \fg(M, \cO_\cM/\cI_\cM^d)
\]
is a Kan fibration for every $d$, with the fiber determined by the Maurer-Cartan simplicial set of the abelian dg Lie algebra $\eta^\ast \fg \otimes_{\Omega^\ast_M} \cI_\cM^d/\cI_\cM^{d+1}$, which is simply the Dold-Kan simplicial set for a cochain complex. (See the discussion before Lemma C.8.) Then $B\Psi$ induces a map between the fiber sequences  
\[
\xymatrix{
\fib(\bB \fg(q_{d})) \ar[r] \ar[d]^{B\Psi_\fib} & \bB \fg(M, \cO_\cM/\cI_\cM^{d+1}) \ar[r] \ar[d]^{B\Psi_{d+1}} & \bB \fg(M, \cO_\cM/\cI_\cM^d) \ar[d]^{B\Psi_{d}} \\
\fib(\bB \fh(q_{d}))\ar[r] & \bB \fh(M, \cO_\cM/\cI_\cM^{d+1}) \ar[r] & \bB \fh(M, \cO_\cM/\cI_\cM^d)
}
\]
and the base is a homotopy equivalence by induction. Hence, it is sufficient to check that the map of fibers
\[
\Psi_\fib \co \eta^\ast \fg \otimes_{\Omega^\ast_M} \cI_\cM^d/\cI_\cM^{d+1} \to \eta^\ast \fh \otimes_{\Omega^\ast_M} \cI_\cM^d/\cI_\cM^{d+1}
\]
is a quasi-isomorphism. Consider the filtration on the fibers induced by the ideal $\Omega^{\ge 1}_M$. Note that $\Psi_\fib$ is a filtration-respecting map of filtered cochain complexes.  We have an isomorphism
\[
\Gr (\eta^\ast \fg \otimes_{\Omega^\ast_M} \cI_\cM^d/\cI_\cM^{d+1}) \cong
\eta^{-1} (\Gr \fg) \otimes_{\eta^{-1} \Omega^\sharp_X} \Omega^\sharp_M \otimes_{\Omega^\sharp_M} \cI_\cM^d/\cI_\cM^{d+1} ,
\]
and similarly when $\fh$ is substituted for $\fg$. Since the $\L8$ map $\psi \co \fg \to \fh$ is a weak equivalence, it induces a quasi-isomorphism at the level of associated gradeds, and hence a quasi-isomorphism of the associated gradeds of the above fibers. Thus, the map of spectral sequences arising from $\Psi_\fib$ is a quasi-isomorphism on the first page and so $\Psi_\fib$ is a quasi-isomorphism.

Proving the second claim is similar. Suppose we know that $\bB \Psi \co \bB \fg \to \bB \fh$ is a weak equivalence of derived stacks. Hence, on any $M_{sm}$---the nil dg manifold $(M,\cinf_M)$ associated to a smooth manifold $M$---we must have an isomorphism of sets
\[
f \circ -\co \bB \fg(M_{sm}) = \op{Maps}(M,X) \to \op{Maps}(M,Y) = \bB \fh(M_{sm}),
\]
and hence $f \co X \to Y$  must be a diffeomorphism. Again, without loss of generality, we can assume $Y = X$ and $f$ is the identity. It remains to show that $\Gr \psi$ is a weak equivalence of $\L8$ algebras. We will consider the nil dg manifold $X_{dR} = (X, \Omega_X)$, whose canonical tower of algebras is
\[
\Omega_X^* \to \Omega_X^*/\Omega_X^{n -1} \to \cdots \to \Omega_X^*/\Omega^{\geq 1}_X \cong \cinf_X,
\]
where $n = \dim X$. Moreover, we will only consider the Maurer-Cartan simplicial sets over the identity map $X \to X$. In that case, our hypothesis implies that for every $d$, $\bB \Psi$ induces a weak homotopy equivalence
\[
DK(\Omega^d_X \fg/\Omega^{d+1}_X \fg) \simeq DK(\Omega^d_X \fh/\Omega^{d+1}_X \fh)
\]
of the simplicial sets associated by the Dold-Kan correspondence to these abelian dg Lie algebras, which describe the fibers as we work up the tower for $X_{dR}$. As $\Gr \fg = \bigoplus_d\Omega^d_X \fg/\Omega^{d+1}_X \fg$, we see that the map of cochain complexes, induced from $\psi$, from $\Gr \fg$ to $\Gr \fh$ must be a quasi-isomorphism. 
\end{proof}

\subsection{Vector bundles on $\L8$ spaces and shifted symplectic structures}
\label{l8vb}

$\L8$ spaces admit straightforward generalizations of many geometric constructions. We begin by recalling the relevant notion of vector bundle and then of shifted symplectic structures.

\subsubsection{}

We now introduce a category $\VB(X,\fg)$ of vector bundles on an $\L8$ space~$(X,\fg)$.

\begin{definition}\label{defn:vb}
Let $(X, \fg)$ be an $\L8$ space.  A {\it vector bundle} on $(X,\fg)$ is a $\ZZ$-graded vector bundle $\pi\co V \to X$ where the sheaf of smooth sections $\cV$ over $X$ is equipped with the structure of an $\Omega^\sharp_X$-module and where the direct sum of sheaves $\fg \oplus \cV$ is equipped with the structure of a curved $\L8$ algebra over $\Omega^*_X$, which we denote $\fg \ltimes \cV$, such that
\begin{enumerate}
\item[(1)] the maps of sheaves given by inclusion $\fg \hookrightarrow  \fg \ltimes \cV$ and by the projection $\fg \ltimes \cV \to \fg$ are maps of $\L8$ algebras, and
\item[(2)] the Taylor coefficients $\ell_n$ of the $\L8$ structure vanish on tensors containing two or more sections of~$\cV$.
\end{enumerate}
The {\it sheaf of sections of $\cV$ over $(X,\fg)$} means $\widehat{C}^* (\fg, \cV[1])$, the sheaf on $X$ of dg $\widehat{C}^*(\fg)$-modules given by the Chevalley-Eilenberg cochains of $\cV$ as a $\fg$-module. The {\it total space} for the vector bundle $\cV$ over $(X,\fg)$ is the $\L8$ space $(X, \fg \ltimes \cV)$.
\end{definition}

\begin{remark}\label{l8beck}
This definition is just a version of a module over an $\L8$ algebra, where we require the underlying module to come from a vector bundle on the manifold $X$ over which the curved $\L8$ algebra $\fg$ lives. Under the usual correspondence between modules and abelian group objects, this definition amounts to giving an abelian group object in the category of $\L8$ spaces over $(X,\fg)$, i.e., the over-category $\linfcatbf(X)_{/(X,\fg)}$. In other words, we are just deploying the notion of a Beck module. Compare with Remark \ref{lalgdbeck}. (For a quick and clear explanation of how Beck modules compare to the more familiar notions, see Lemma 1.3 of \cite{BergerMoerdijk}.) 
\end{remark}

For example, the tangent bundle to $(X, \fg)$ is given by $\fg[1]$ equipped with the adjoint action of $\fg$. Dually, the cotangent bundle is given by $\fg^\vee [-1]$ equipped with the coadjoint action.  It follows that $k$-forms on $(X, \fg)$ are given~by 
\[
\Omega^k_{(X,\fg)} = \widehat{C}^* (\fg , (\Lambda^k \fg) [-k]),
\]
as discussed in \cite{GGLoop}.

\begin{definition}
Let $\cV$ and $\cW$ be vector bundles on the $\L8$ space $(X, \fg)$. A {\em map of vector bundles} from $\cV$ to $\cW$ is a map $\phi \co \widehat{C}^* (\fg, \cV[1]) \to \widehat{C}^* (\fg , \cW[1])$ of dg $\widehat{C}^* (\fg)$-modules. 
\end{definition}

We wish to pinpoint the appropriate notion of weak equivalence of vector bundles. A map of vector bundles $\phi$ induces a map $\Gr \phi \co \Gr \widehat{C}^* (\fg, \cV[1]) \to \Gr \widehat{C}^* (\fg , \cW[1])$ of dg $\Gr \widehat{C}^* (\fg)$-modules, with respect to the filtration by powers of the nilpotent ideal. Since $\Gr \fg$ is a non-curved $\L8$ algebra, there is a natural decreasing filtration 
\[
\cdots \hookrightarrow F_{\Gr}^1  \Gr \widehat{C}^* (\fg, \cV[1]) \hookrightarrow  F_{\Gr}^0  \Gr \widehat{C}^* (\fg, \cV[1]) =  \Gr \widehat{C}^* (\fg, \cV[1]) 
\]
on such a module $\Gr \widehat{C}^* (\fg, \cV[1])$ by symmetric powers of the dual of $\fg$. (Compare to how, for an ordinary Lie algebra $\fg$, one filters $C^*(\fg,V)$ by $F^k C^*(\fg,V) = \Sym^{\geq k}(\fg^\vee[-1]) \otimes V$.)
The quotient of $\Gr \widehat{C}^* (\fg, \cV[1])$ by the first piece of the filtration has underlying graded vector space $\cV[1]$, which is thus equipped with a $\Omega^\sharp$-linear differential. Let $\phi_\fib$ denote the map
\[
\Gr \widehat{C}^* (\fg, \cV[1])/ F_{\Gr}^1  \Gr \widehat{C}^* (\fg, \cV[1]) \to \Gr \widehat{C}^* (\fg, \cW[1])/ F_{\Gr}^1  \Gr \widehat{C}^* (\fg, \cW[1]) 
\]
induced by~$\Gr\phi$.

\begin{definition}\label{vbwe}
Let $\phi \co \widehat{C}^* (\fg, \cV[1]) \to \widehat{C}^* (\fg , \cW[1])$ be a map of vector bundles on $(X, \fg)$.  Then $\phi$ is a {\em weak equivalence} if $\phi_\fib$ is a quasi-isomorphism.
\end{definition}

Note that a vector bundle map $\phi$ induces a map of $\L8$ spaces on the total spaces, and it is a weak equivalence of vector bundles if and only if the map of total spaces is a weak equivalence. This notion of weak equivalence is strictly stronger than requiring the map of sections $\phi$ to be a filtered quasi-isomorphism.

\subsubsection{}

Recall that a symplectic form is a 2-form that is nondegenerate and closed. This definition works perfectly well in the derived setting, so long as one recognizes that being closed---i.e., being annihilated by the differential of the de Rham complex---is data and not a property.

Let $\Omega^{2,cl}_{(X,\fg)}$, the complex of {\em closed 2-forms} on the $\L8$ space, be the totalization of the double complex
\[
\Omega^2_{(X,\fg)} \xto{d_{dR}} \Omega^3_{(X,\fg)} \xto{d_{dR}} \Omega^4_{(X,\fg)} \xto{d_{dR}} \cdots.
\]
A {\em closed 2-form} is a cocycle in this complex. Every element $\omega$ of $\Omega^{2,cl}_{(X,\fg)}$ has an underlying 2-form $i(\omega)$ by taking its image under the truncation map $i \co \Omega^{2,cl}_{(X,\fg)}\to\Omega^2_{(X,\fg)}$.

\begin{definition}
An {\em $n$-shifted symplectic form} on an $\L8$ space $(X,\fg)$ is a closed 2-form $\omega$ of cohomological degree $n$ such that the induced map $i(\omega) \co T_{(X,\fg)} \to T^\vee_{(X,\fg)}[n]$ is a quasi-isomorphism.
\end{definition}

\subsection{Discussion of related work}\label{sect:related}

In the last few decades, derived geometry, particularly of an algebraic flavor, has developed rapidly,
and we make no attempt here to place $L_\infty$ spaces into that broader context.
There are strong similarities, however, with two recent, substantial works, \cite{GR} and \cite{CPTVV},
that we would like to sketch.

A central tenet  of \cite{GR} and \cite{CPTVV} is that, to first approximation, derived geometry is just affine geometry over the de Rham stack. 
Note that any derived stack (or even prestack) $X$ maps canonically to its underlying de Rham stack $X_{dR}$,
which does not see nilpotent or derived directions.
In a sense, $X_{dR}$ just sees the macroscopic structure of $X$ and not its very local geometry;
 it only sees the ``topology.''
The idea then is that objects over a de Rham stack should be thought of as local systems on $X$,
and \cite{GR,CPTVV} provide precise theorems in this direction. 
Hence, constructions over the de Rham stack are fiberwise constructions over $X$ with the extra data of a flat connection relating the fibers.

A similar idea underlies the definition of $\L8$ spaces.
Given a smooth manifold, we have the associated {\it de Rham space} $X_{dR} = (X, \Omega^\ast_X)$. The space $X_{dR}$ presents the de Rham stack in our smooth setting. 
The definition of an $L_\infty$ space $(X,\fg)$ determines a derived stack $\bB \fg$ living over $X_{dR}$, thereby porting the overarching paradigm to a smooth setting.
On the other hand, we have not developed an {\em a priori} global derived geometry, in contrast to the foundational work of To\"en-Vezzosi and Lurie, so we cannot take advantage of these two perspectives, as done in~\cite{GR,CPTVV}.

In Section \ref{sect:lieFMP} we discuss how recent work in derived geometry connects with Lie algebroids.

\section{Recollections on Lie algebroids}

We will give a brief overview of the definitions and constructions from the theory of Lie algebroids that are relevant to our work here. Standard references for Lie algebroids include Mackenzie \cite{Mackenzie} and Rinehart \cite{Rinehart}; we also recommend the article of Fernandes~\cite{Fernandes}.

\subsection{The objects of study}\label{sect:Liedefn}

\begin{definition}
A {\it Lie algebroid} on a smooth manifold $X$ is a vector bundle $L \to X$ equipped with the structure of a Lie algebra on its sheaf of smooth sections and an {\it anchor map} $\rho \co L \to T_X$, which is a map of vector bundles, such~that
\begin{enumerate}
\item the map on sections induced by $\rho$ is a map of Lie algebras and
\item for $x,y \in \Gamma (L)$ and $f \in C^\infty_X$, we have the Leibniz rule
\[
[x, fy] = f [x,y] + (\rho(x) f) y.
\]
\end{enumerate}
We will use $\cL$ to denote the sheaf of smooth sections of~$L$.
\end{definition}

Two classes of examples give a sense of the range of Lie algebroids. At the purely algebraic end, note that every Lie algebra is a Lie algebroid over the point. At the geometric end, a regular foliation on a smooth manifold gives an example, where the anchor map is the inclusion of the subbundle into the tangent bundle. In general, the image of the anchor map $\rho$ is a (possibly singular) foliation, and the kernel of the anchor map (on sections) is a $\cinf$-linear Lie algebra. Thus, a generic Lie algebroid is a complicated mix of a foliation and $\cinf$-linear Lie algebra.

\begin{definition}
A {\it map of Lie algebroids} $F\co L \to L'$ is a map of vector bundles
\[
\xymatrix{
L \ar[d] \ar[r]^{\phi} & L' \ar[d] \\
X \ar[r]^{f} & Y
}
\]
such that 
\begin{enumerate}
\item the anchor maps intertwine $\rho' \circ \phi = df \circ \rho$ and
\item the natural map from sections of $L$ to sections of the pullback bundle $f^{-1} L'$ is a map of Lie algebras.
\end{enumerate}
We say a map is {\it base-fixing} if the map of manifolds $f$ is the identity.
\end{definition}

There are thus two categories of interest to us. Let $\lalgd$ denote the category of all Lie algebroids and all maps of Lie algebroids. For each manifold $X$, there is also the category $\lalgdbf(X)$ whose objects are Lie algebroids on $X$ and whose morphisms are the base-fixing maps thereof. (These are 1-categories, not $(\infty,1)$-categories.)

Note that on a given manifold $X$, there are two distinguished Lie algebroids: the trivial algebroid $L = 0$ and the tangent bundle $\id\co L = T \to T$ with the identity as the anchor map. Every other Lie algebroid sits between them. Succinctly, we might say that they are the initial and terminal objects of $\lalgdbf(X)$, respectively.

\subsubsection{}

There are natural dg generalizations of Lie algebroids. We will work with the following. 

\begin{definition}
A {\it dg Lie algebroid} is a $\ZZ$-graded vector bundle $L \to X$ of total finite rank whose sheaf $\cL$ of graded smooth sections is equipped with a $\cinf_X$-linear differential, the structure of a dg Lie algebra (over the constant sheaf $\CC_X$), and an anchor map $\rho\co \cL \to \cT_X$ of dg Lie algebras such that
\[
[x, fy] = f [x,y] + (\rho(x) f) y
\]
for $x,y \in \cL$ and $f \in C^\infty_X$. (In other words, ignoring the differential on $\cL$, we have a graded Lie algebroid, and the differential is compatible with the bracket by being a derivation of the graded Lie algebra.) A {\it map of dg Lie algebroids} is a map of the underlying graded Lie algebroids that is also a cochain map.
\end{definition}

Again we have two categories of interest: $\dglalgd$ and $\dglalgdbf(X)$. We say a map of dg Lie algebroids $(f, \phi)$ is a {\it weak equivalence} if the underlying map $f$ is a diffeomorphism and if $\phi$ induces a quasi-isomorphism of sheaves of cochain complexes. Consequently, the categories $\dglalgd$ and $\dglalgdbf(X)$ become categories with weak equivalences. (See Vezzosi's work \cite{Vez15} for how to deal properly with the associated $\infty$-categories; he works, of course, in the setting of derived algebraic geometry.)

\begin{remark}
Our constructions seem to work without difficulty for reasonable notions of $\L8$ algebroid, but in the literature there seems to be some variation in the meaning of this term. We indicate in Remark \ref{l8algebroid} a definition that admits the easiest direct modification of our arguments.
\end{remark}

\subsection{The Chevalley-Eilenberg complex of a Lie algebroid}

To any Lie algebroid $\rho \co L \to T_X$ there is an associated commutative dg algebra. We will call it the {\it Chevalley-Eilenberg cohomology complex of $L$} and denote it $\calgd(L)$, because it is modeled on the Chevalley-Eilenberg cochain complex of a Lie algebra. It is often also called the {\it de Rham complex of $L$} because for the Lie algebroid $\id\co L = T \to T$, we have $\calgd(L) = \Omega^*_X$, the usual de Rham complex. 

The complex is constructed as follows. Let $L^\vee$ be the dual vector bundle to $L$ and consider the map $d_L \co \Gamma (X,\Lambda^m L^\vee) \to \Gamma (X,\Lambda^{m+1} L^\vee)$ given~by
\begin{align*}
\MoveEqLeft[8] (d_L \alpha)(x_0 , \dotsc, x_m) = \frac{1}{m+1} \sum_{k=0}^{m+1} (-1)^k \rho (x_k) \alpha (x_0, \dotsc, \widehat{x_k} , \dotsc , x_m )\\[1ex]
& + \frac{1}{m+1} \sum_{k < l} (-1)^{k+l+1} \alpha ([x_k , x_l], x_0 , \dotsc , \widehat{x_k} , \dotsc , \widehat{x_l} , \dotsc , x_m ).
\end{align*}
Define $\calgd(L)$ to be the cochain complex 
\[
\Gamma (X, \Lambda^0 L^\vee) \xto{d_L} \Gamma (X, L^\vee)\xto{d_L} \cdots \xto{d_L} \Gamma (X,\Lambda^{n-1} L^\vee) \xto{d_L} \Gamma (X,\Lambda^n L^\vee)
\]
where $n = \op{rk}(L)$. (Note that the first term is $\cinf(X)$.) We will let $H^\ast_L(X)$ denote the cohomology of~$\calgd(L)$.  

Notice that $\calgd(L)$ naturally receives a map of commutative dg algebras from $\Omega^*_X$, arising from the identity on $\cinf(X) = \Gamma (X, \Lambda^0 L^\vee)$ and the dual to the anchor map $\rho^\vee \co \Omega^1(X)  \to \Gamma (X, L^\vee)$. In particular, we have a map $H^*_{dR}(X)~\to~H^*_L(X)$. 

These constructions are local in nature: we can consider the Lie algebroid restricted to any open subset of $X$, and so $\calgd(L)$ provides  a sheaf of commutative dg algebras on $X$. In the case of $T_X$ as a Lie algebroid, we recover the de Rham complex $\Omega^*_X$ as a sheaf. We will use this notation $\calgd(L)$ to refer to this sheaf, somewhat abusively. 

As in the case of dg Lie algebras, the definition of $\calgd (L)$ canonically extends to dg Lie algebroids.  The dg Lie algebroid structure of $L$ defines an internal differential on each graded vector space appearing in the complex above, so that we have a double complex. Thus, for a dg Lie algebroid $L$, we define $\calgd (L)$ to be the associated total complex.

\section{Lie algebroids as $\L8$ spaces}

The first important result of this paper is that every Lie algebroid $\rho \co L \to T_X$ has a naturally associated $\L8$ space $(X, \op{enh}(L))$---hence a derived stack---in the sense described in section \ref{sect:L8spaces}. In this section, we develop this result in two stages. First, we explain how the $\infty$-jet bundle of $\calgd(L)$ provides a curved $\L8$ algebra over $\Omega^*_X$, which is an explicit construction in differential geometry. Second, we verify the functoriality of this construction, which involves categorical issues. A precise statement of the main result appears there.

\subsection{The construction of $\op{enh}(L)$}\label{sect:construction}

The crucial tool here is the functor $J$ that assigns to a vector bundle$E$, its $\infty$-jet bundle $J(E)$. In appendix \ref{app}, we provide proofs and references for the facts we use here. Our arguments amount to a variation on constructions often described as Gelfand-Kazhdan formal geometry or Fedosov resolutions.

We begin with some preliminaries. For any vector bundle $V$, let $\calsym(V)$ denote the sheaf of smooth sections of the filtered vector bundle $\csym(V) = \lim_k \Sym^{\leq k}(V)$. (It is also fruitful to view $\csym(V)$ as a pro-vector bundle.) Recall that every $\infty$-jet bundle $J(E)$ admits a non-canonical isomorphism $\sigma_E \co \csym(T^\vee_X) \ot E \to J(E)$ of filtered vector bundles on $X$. In particular, the $\infty$-jet bundle of the trivial line bundle, which we denote simply $J$, admits a non-canonical isomorphism of filtered algebras to $\csym(T^\vee_X)$. Moreover, the sheaf of smooth sections $\cJ(E)$ of $J(E)$ is a module over $\cJ$, which is the sheaf of $\infty$-jets of smooth functions. Hence, we can ask for compatible isomorphisms $\sigma_0 \co \csym(T^\vee_X) \to J$ and $\sigma_E \co \csym(T^\vee_X) \ot E \to J(E)$ so that the natural module structures intertwine.

We now combine these constructions in the case of interest.
 
\begin{lemma}\label{lem:algebra}
Fix an isomorphism $\sigma_0 \co \csym(T^\vee_X) \to J$ and fix a compatible isomorphism $\sigma_{L^\vee} \co \csym(T^\vee_X)\ot L^\vee \to J(L^\vee)$. Then we have an isomorphism 
 \[
\calsym \left ( T_X^\vee \oplus L^\vee [-1] \right ) \xrightarrow{\cong} \cJ (\Sym(L^\vee[-1]))
 \] 
 of $C^\infty_X$-algebras.
\end{lemma}
 
 \begin{proof}
 As the functor $V \mapsto \cJ(V)$ is symmetric monoidal by proposition \ref{prop:jetsm}, we see that the algebra $\Sym(L^\vee[-1])$ in vector bundles gets mapped to the algebra $\cJ(\Sym(L^\vee[-1])) \cong \Sym_\cJ(\cJ(L^\vee)[-1])$.~But
 \[
 \Sym_\cJ(\cJ(L^\vee)[-1]) \cong \Sym_{\calsym(T^\vee_X)}(\calsym(T^\vee_X) \ot_{\cinf} \cL^\vee[-1])
 \]
 by our choice of isomorphisms. As base change commutes with taking free algebras, we~have
 \begin{align*}
 \Sym_{\calsym(T^\vee_X)}(\calsym(T^\vee_X) \ot_{\cinf} \cL^\vee[-1]) &\cong \calsym(T^\vee_X) \ot_{\cinf} \Sym_{\cinf}(\cL^\vee[-1]) \\
 &\cong \calsym(T^\vee_X) \ot_{\cinf} \calsym(L^\vee[-1])\\
 &\cong \calsym(T^\vee_X \oplus L^\vee[-1]), 
 \end{align*}
 as desired.
 \end{proof}
 
Recall that for a Lie algebroid $L$, the Chevalley-Eilenberg complex $\calgd(L)$ has underlying graded algebra $\calsym(L^\vee[-1])$. Its differential $d_L$ is a differential operator, so that taking $\infty$-jets, we obtain a sheaf of  commutative dg algebras
 \[
(\cJ(\Sym(L^\vee[-1])), \cJ(d_L)).
 \]
We will denote this sheaf of commutative dg algebras by $\cJ(\calgd(L))$, for brevity's sake. As every $\infty$-jet bundle $J(E)$ has a canonical flat connection (proposition \ref{prop:grconn}), we can take the de Rham complex of $\cJ(\calgd(L))$ to obtain a commutative dg algebra over $\Omega^*_X$. We denote this sheaf of commutative dg $\Omega^*_X$-algebras by~$dR(J(\calgd(L)))$. 

We now provide our primary construction.

 \begin{theorem}\label{thm:construction}
Let $\rho \co L \to T_X$ be a Lie algebroid over a smooth manifold $X$. For any choice of compatible splittings $\sigma = (\sigma_0, \sigma_{L^\vee})$, there exists a curved $\L8$ algebra $\op{enh}(L)^\sigma$ over $\Omega^\ast_X$ such~that
 \begin{enumerate}
 \item $\op{enh}(L)^\sigma \cong \Omega^\sharp_X (T_X[-1] \oplus L)$ as $\Omega^\sharp_X$-modules;
 \item $\widehat{C}^* ( \op{enh}(L)^\sigma) \cong dR(J(\calgd(L)))$ as commutative $\Omega^\ast_X$-algebras; and
 \item the map sending a section to its~$\infty$-jet, 
 \[
 j_\infty \co \calgd(L) \hookrightarrow dR (J(\calgd(L))) \cong \widehat{C}^* (\op{enh}(L)^\sigma),
 \]
 defines a quasi-isomorphism of  $\Omega^\ast_X$-algebras.
 \end{enumerate}
 \end{theorem}
 
\begin{remark}
The final claim asserts that $\widehat{C}^* (\op{enh}(L)^\sigma)$ is a semi-free resolution of $\calgd(L)$ as an algebra over $\Omega^*_X$. Inasmuch as we are developing an approach to derived geometry over the base ring $\Omega^*_X$, working with this algebra will provide the ``homotopically correct'' answers to questions about~$\calgd(L)$.
\end{remark}
 
\begin{remark}
The theorem generalizes the construction in \cite{GGLoop} in the case $L = 0$, as well as that of Costello in \cite{CosWG} the case of a complex foliation~$T^{0,1}_X\hookrightarrow~T_X~\otimes~\CC$. 
\end{remark} 
 
\begin{proof}[Proof of Theorem \ref{thm:construction}]
As shown by Lemma \ref{lem:algebra}, a choice of compatible splittings provides an isomorphism of $\Omega^\sharp_X$-modules
 \[
 \sigma \co \Omega^\sharp_X(\csym \left ( T_X^\vee \oplus L^\vee [-1] \right )) \to \Omega^\sharp_X(J (\Sym(L^\vee[-1]))),
\]
just by tensoring the isomorphism in the lemma over $\cinf_X$ with $\Omega^\sharp_X$. Now the right hand side has the canonical differential $\nabla_{\Sym(L^\vee[-1])} + J(d_L)$, which is the sum of the flat connection on the jet bundle $J (\Sym(L^\vee[-1]))$ and the operator $J(d_L)$ . Thus, the left hand side inherits a differential $\dd_{\op{enh}(L)}$. Compatibility of the algebra structures ensures that this transferred differential is also a derivation. Hence claim (2) amounts to interpreting this completed commutative dg algebra 
\[
(\Omega^*_X(\csym \left ( T_X^\vee \oplus L^\vee [-1] \right )),\dd_{\op{enh}(L)})
\]
as the Chevalley-Eilenberg cochains of some curved $\L8$ algebra $\op{enh}(L)$, which is immediate. Claim (1) amounts to recognizing the underlying vector bundle for $\op{enh}(L)$ as $L~\oplus~T_X[-1]$.
 
Claim (3) follows from a standard jet bundle argument. See proposition \ref{prop:grconn} or, e.g.,  Proposition 3.2 of~\cite{CDH}. 
\end{proof}
 
Examining the proof, one recognizes that the argument applies verbatim to a dg Lie algebroid. 

\begin{cor}\label{cor:dgversion}
For $\rho \co L \to T_X$ a dg Lie algebroid and a choice of compatible splitting $\sigma$, there is a curved $\L8$ algebra $\op{enh}(L)^\sigma$ over $\Omega^\ast_X$ such that
 \begin{enumerate}
 \item $\op{enh}(L)^\sigma \cong \Omega^\sharp_X (T_X[-1] \oplus L^\sharp)$ as $\Omega^\sharp_X$-modules;
 \item $\widehat{C}^* ( \op{enh}(L)^\sigma) \cong dR(J(\calgd(L)))$ as commutative $\Omega^\ast_X$-algebras; and
 \item the map sending a section to its $\infty$-jet, 
 \[
 j_\infty \co \calgd(L) \hookrightarrow dR (J(\calgd(L))) \cong \widehat{C}^* (\op{enh}(L)^\sigma),
 \]
 defines a quasi-isomorphism of  $\Omega^\ast_X$-algebras.
 \end{enumerate}
 \end{cor}

\begin{remark}
The construction works for an arbitrary dg manifold, and it amounts to providing a Fedosov resolution $F_\cA$ of the structure sheaf $\cA$ and then taking the total complex of the de Rham complex of that Fedosov resolution. (As the structure sheaf is a dg algebra, there is an internal differential in $\Omega^k(F_A)$ for every $k$.)  Our goal, as the next sections make clear, is not to provide interesting objects, but to have categorical statements. (The challenge in derived geometry is typically to have good control and understanding of the morphisms, particularly weak equivalences.) Hence, we have focused on Lie algebroids, where the categorical framework is better developed than the general theory of dg manifolds. 
\end{remark}

\subsection{A functorial statement I: the base-fixing case}\label{sect:functorial1}

In the construction above, we relied on a choice of splittings for the relevant jet bundles. Via the isomorphisms produced by the construction, we can view a change of splitting as providing an isomorphism of curved $\L8$ algebras (not just quasi-isomorphism!). Thus, in a certain sense, the splitting does not matter. We now develop a precise version of this idea.

In this section we will work over a fixed base manifold~$X$.

\begin{definition}
Let $\dglalgd^\sigma$ denote the category whose objects are pairs $(L,\sigma)$, where $L$ is a dg Lie algebroid and $\sigma = (\sigma_0,\sigma_1)$ is a compatible pair of splittings, and whose morphisms are simply maps of the underlying dg Lie algebroids (i.e., do not depend on the splittings in any way). Let $\dglalgd^\sigma_{\rm bf}(X)$ denote the category where we fix the base manifold to be $X$ and only allow base-fixing morphisms.
\end{definition}

\begin{lemma}
There is a faithful functor
\[
\begin{array}{ccccc}
\op{enh}^\sigma \co & \dglalgd^\sigma_{\rm bf}(X) &\to &\L8\text{-}\mathsf{space}_{\rm bf}(X),\\ 
& \quad (L, \sigma) &\mapsto & (X,\op{enh}(L)^\sigma),
\end{array}
\]
provided by the construction of Theorem~\ref{thm:construction}.
\end{lemma}

\begin{proof}
We have specified what to assign to objects, but we need to specify the rest of the functor. In other words, for each base fixing map of dg Lie algebroids $L \to L'$ over $X$, we need to provide a map $(X, \op{enh}(L)) \to (X, \op{enh}(L'))$ of $\L8$ spaces, and then we need to verify that our construction respects composition of maps and also sends an identity map to an identity map.

Let $\phi \co L \to L'$ be a base fixing map of dg Lie algebroids on $X$.  There is  a canonical map of dg $\Omega^*_X$-algebras $\phi^\vee \co \calgd(L') \to \calgd(L)$, which induces a map $dR(J(\phi^\vee)) \co dR(J(\calgd(L'))) \to dR(J(\calgd(L)))$. Using the splittings before and after this map, we obtain a map
\[
\widehat{C}^*(\op{enh}(L'))\underset{\cong}{\xto{\sigma_{L'}}} dR(J(\calgd(L'))) \xto{dR(J(\phi^\vee))} dR(J(\calgd(L))) \underset{\cong}{\xto{\sigma_L^{-1}}} \widehat{C}^*(\op{enh}(L))
\]
of filtered commutative dg $\Omega^*_X$ algebras, and hence a map of the associated $\L8$ spaces. From this explicit formula for the map of $\L8$ spaces, it is clear that the identity goes to the identity: the inner map $dR(F(J(\id^\vee)))$ is simply the identity, so that the outer maps cancel because they are given by the splitting and its inverse. 

Now let $L \xto{\phi} L' \xto{\psi} L''$ be a composition of base fixing maps of dg Lie algebroids. At the level of dg $\Omega^\ast_X$ algebras, we have
\[
dR(J((\psi \circ \phi)^\vee)) = dR(J(\phi^\vee)) \circ \sigma_{L'} \circ \sigma_{L'}^{-1} \circ dR(J(\psi^\vee)),
\]
so that $\op{enh}^\sigma (\psi \circ \phi) = \op{enh}^\sigma (\psi)~\circ~\op{enh}^\sigma (\phi)$.
\end{proof}

\begin{lemma}
The forgetful functor $\op{F} \co \dglalgd^\sigma_{\rm bf}(X) \to \dglalgd_{\rm bf}(X)$ is an equivalence of categories. 
\end{lemma}

\begin{proof}
The forgetful functor is clearly essentially surjective, and it is fully faithful, by construction. Hence by Theorem 1 of IV.4 of \cite{MacLane}, it is an equivalence.
\end{proof}

Putting these lemmas together, we obtain the result we desire.

\begin{prop}\label{prop:functor}
The construction \ref{thm:construction} produces a functor $\op{enh} \co \dglalgd_{\rm bf}(X) \to \L8\text{-}\mathsf{space}_{\rm bf}(X)$ that is unique up to natural isomorphism.
\end{prop}

\begin{proof}
Any choice of ``inverse'' $\op{G} \co \dglalgd_{\rm bf} (X) \to \dglalgd^\sigma_{\rm bf} (X)$ to $\op{F}$ produces a functor $\op{enh}^\sigma \circ \op{G} \co \dglalgd_{\rm bf} (X) \to \linfcat_{\rm bf}(X)$. Moreover, for any two choices $\op{G}$ and $\op{G'}$, there is a natural isomorphism $\op{enh}^\sigma \circ \op{G} \Rightarrow \op{enh}^\sigma \circ \op{G'}$. Hence, this functor $\op{enh}^\sigma \circ \op{G}$ is unique up to natural isomorphism.
\end{proof}

Recall that both $\dglalgd (X)$ and $\L8\text{-}\mathsf{space}_{\rm bf}(X)$ are categories with weak equivalences. The functor $\op{enh}$ is compatible with this structure; more precisely, we have the following.

\begin{prop}\label{prop:bfhpty}
The functor $\op{enh} \co \dglalgd_{\rm bf}(X) \to \L8\text{-}\mathsf{space}_{\rm bf}(X)$ preserves weak equivalences.
\end{prop}

\begin{proof}
Let $\phi \co L \to L'$ be a weak equivalence of dg Lie algebroids over $X$. Fix a splitting of the jet sequence and further choose splittings in order to obtain isomorphisms of $\Omega^\sharp$-modules
\[
\op{enh}(L) \cong \Omega^\sharp_X (T_X[-1] \oplus L), \text{ and }
\op{enh}(L') \cong \Omega^\sharp_X (T_X[-1] \oplus L').
\]
Since $\phi$ is a strict map of vector bundles, the induced map
\[
\op{enh} (\phi)_1 \co \Omega^\sharp_X (T_X[-1] \oplus L) \to \Omega^\sharp_X (T_X[-1] \oplus L')
\]
is simply given by 
\[
\op{enh} (\phi)_1 = \id_{\Omega^\sharp_X}  \otimes (\id_{T_X[-1]} \oplus \phi).
\]
Hence, $\Gr (\op{enh} (\phi)_1)$ is a quasi-isomorphism as $\phi$ itself is a quasi-isomorphism. 
\end{proof}

\subsection{A functorial statement II: the general case}

We now prove that our construction is actually functorial with respect to arbitrary maps of Lie algebroids, not just base-fixing maps.

\begin{prop}\label{fullfunctoriality}
The construction of Theorem \ref{thm:construction} defines a faithful functor
\[
\begin{array}{ccccc}
\op{enh} \co & \dglalgd  &\to &\L8\text{-}\mathsf{space}\\ 
&  (\rho \co L \to T_X) &\mapsto & (X,\op{enh}(L))
\end{array}
\]
unique up to natural isomorphism.
\end{prop}

\begin{proof}
As noted in the proof of the base-fixing case, the functor is already defined on objects so the remaining work is to describe how the functor behaves on maps, check associativity, and so on. We will also fix a choice of splittings for each dg Lie algebroid to produce an $\L8$ space, but the same arguments as in base-fixing case ensure that the choices are essentially irrelevant.

Now let $\rho_L \co L \to T_X$ and $\rho_K \co K \to T_Y$ be Lie algebroids equipped with splittings $\sigma_L$ and $\sigma_K$ of their respective jet sequences. Let $F =(f, \varphi) \co L \to K$ be a morphism of dg Lie algeboids. We need to produce a map of $\L8$ spaces 
\[
\op{enh}\, F = (f, \psi_F) \co (X, \op{enh}(L)) \to (Y, \op{enh}\,K), 
\]
notably a map of filtered commutative dg  $\Omega^\ast_X$-algebras
\[
\psi_F \co  \widehat{C}^\ast ( f^\ast \op{enh}\,K) \to \widehat{C}^\ast ( \op{enh}(L)).
\]
Recall that the splittings induce isomorphisms
\[
\sigma_{K} \co \widehat{C}^\ast (\op{enh}\,K) \xto{\cong} dR_Y ( J_Y (\calgd(K))) \quad \text{ and } \quad \sigma_{L}^{-1} \co dR_X (J_X (\calgd(L))) \xto{\cong} \widehat{C}^\ast (\op{enh}(L)).
\]
As in the base-fixing case, the key is thus to exploit the nice behavior of the jets functor $J$ and then pre- and post-compose by these isomorphisms from the splittings.

Via base change, we have an isomorphism of $\Omega^\ast_X$-algebras
\[
b_F \co \widehat{C}^\ast ( f^\ast \op{enh}\,K) \xto{\cong} f^\ast \widehat{C}^\ast (\op{enh}\,K).
\]
By properties of the functor $J$, specifically Proposition \ref{prop:jetsm} and Proposition \ref{prop:pulljet}, we have a natural map of sheaves of commutative dg algebras
\[
f^{-1} J_Y (\calgd (K)) \to J_X (\calgd (f^{-1} K))
\]
and hence a map of $\Omega^\ast_X$-algebras
\[
p_F \co dR_X (f^{-1} J_Y (\calgd (K))) \to dR_X (J_X (\calgd (f^{-1} K))).
\]
By composition, we thus have an isomorphism
\[
p_F \circ \sigma_K \circ b_F \co \widehat{C}^\ast ( f^\ast \op{enh}\,K) \xto{\cong} dR_X (J_X (\calgd (f^{-1} K))).
%f^\ast dR_Y(J_Y ( \calgd (K))) \xto{\cong} dR_X (f^{-1} J_Y (\calgd (K))) \xto{\cong} dR_X (J_X (\calgd (f^{-1} K)))..
\]
We now use the vector bundle map $\varphi$ in the map $F$ of Lie algebroids. It provides a map $\varphi^\vee \co (f^{-1} K)^\vee \to L^\vee$ on the dual vector bundles, and hence induces a map of $\Omega^\ast_X$-algebras
\[
v_F := dR(J(\varphi^\vee)) \co dR_X (J_X (\calgd (f^{-1} K))) \to dR_X (J_X (\calgd (L))).
\]
We define
\[
\psi_F := \sigma_L^{-1} \circ v_F \circ p_F \circ \sigma_K \circ b_F,
\]
which, aside from the splitting isomorphisms, is determined manifestly by the geometry of the situation.

It remains to verify the $\op{enh}$ respects composition of maps. Now let $G=(g,\phi)$ be a map of Lie algebroids from $\rho_K \co K \to T_Y$ to $\rho_I \co I \to T_Z$. We need to verify that $\op{enh} (G \circ F) = \op{enh} (G) \circ \op{enh} (F)$. Using our notation from above, we see that
\[
\psi_{G \circ F} := \sigma_L^{-1} \circ v_{G \circ F} \circ p_{G \circ F} \circ \sigma_I \circ b_{G \circ F}.
\]
The composition $\op{enh} (G) \circ \op{enh} (F)$ is a bit trickier to describe because one must pull back the map $\psi_G$, which is a map of sheaves on $Y$, to a map of sheaves on $X$
\[
\psi_F \circ f^*\psi_G = \sigma_L^{-1} \circ v_F \circ p_F \circ \sigma_K \circ b_F \circ b_F^{-1} \circ \sigma_K^{-1} \circ f^\ast v_G \circ f^\ast p_G \circ f^\ast \sigma_I \circ f^\ast b_G \circ b_F.
\]
Simplifying, we have
\[
\psi_F \circ f^*\psi_G = \sigma_L^{-1} \circ v_F \circ p_F \circ f^\ast v_G \circ f^\ast p_G \circ  \sigma_I \circ b_{G \circ F}.
\]
Observe next that the map
\[
v_{G \circ F} \circ p_{G \circ F} \co dR_X ((g \circ f)^{-1} J_Z (C^\ast (I))) \to dR_X (J_X (C^\ast L))
\]
can be factored as
\[
v_F \circ p_F \circ f^\ast v_G \circ f^\ast p_G \co dR_X (f^{-1} (g^{-1}(J_Z (C^\ast (I))))) \to dR_X (J_X (C^\ast L)).
\]
Hence
\[
\psi_F \circ f^*\psi_G = \sigma_L^{-1} \circ v_{G \circ F} \circ p_{G \circ F} \circ \sigma_I \circ b_{G \circ F},
\]
which agrees with $\psi_{G \circ F}$, and so we are done.
\end{proof}

As in the preceding section, $\dglalgd$ and $\L8\text{-}\mathsf{space}$ are categories with weak equivalences.  Further, since a weak equivalence of dg Lie algebroids is necessarily a diffeomorphism of the base, we can immediately piggy back off of Proposition~\ref{prop:bfhpty}.

\begin{prop}
The functor $\op{enh} \co \dglalgd \to \L8\text{-}\mathsf{space}$ preserves weak equivalences.
\end{prop}

\subsection{Lie algebroids and formal moduli problems}\label{sect:lieFMP}

That there exists a nice link between Lie algebroids and $\L8$ spaces is not totally surprising;
the parallels are unmistakeable from the outset. 
Indeed, derived deformation theory identifies formal moduli problems and differential graded Lie algebras. 
Initially, this paradigm was formal and algebraic: one studies formal thickenings of a point, which reduces the key issues to algebra. 
It is natural to think in a relative way, replacing the point with a derived scheme or stack $X$ and studying formal thickenings of $X$.
There has been much recent work on this problem in derived algebraic geometry, which we will highlight below. 
Loosely speaking, it identifies relative formal moduli spaces over $X$ with sheaves of (higher) Lie algebras over~$X$.

In a smooth setting, dg Lie algebroids are a natural parametrized version of dg Lie algebras and hence ought to fill one side of such a putative identification. On the other hand, in Costello's formalism, we have shown that $\L8$ spaces {\em do} present families of formal moduli problems. Hence the work of the preceding sections confirms the natural expectation that dg Lie algebroids present relative formal moduli problems.

There is some helpful terminology for discussing various flavors of relative formal moduli problems.
For instance, an $\L8$ space $(X,\fg)$ presents a stack $\bB \fg$ {\em under} $X$ and {\em over} $X_{dR}$, since by construction we have maps $X \to \bB \fg \to X_{dR}$.
These maps have certain properties, however, so we refine the terminology.
We say that $\bB \fg$ is $X$-{\it pointed} since $X$ provides the underlying ``geometric'' points (i.e., without derived fuzz) of $\bB\fg$.
We also say $\bB \fg$ is {\em linear} over $X_{dR}$, because it is described by a sheaf of $\L8$ algebras over $X_{dR}$, and hence is close to being a linear structure rather than having a more complicated fiber structure.

For context, let us outline a few of the highlights from algebraic geometry. Unless otherwise noted, $X$ will denote a scheme or stack over a field $\mathbb{K}$.
\begin{enumerate}
\item For $X$ a K\"ahler manifold, Kapranov \cite{Kap} gave the first example of a $X_{dR}$-linear $\L8$ algebra in describing the the homotopy-Lie algebra structure of $T_X[-1]$ via the Atiyah bracket.
\item Hennion \cite{Hennion} generalized Kapranov's results to the level of derived Artin stacks in characteristic zero. He showed that there was an adjunction between $X$-pointed, $X$-linear formal moduli problems and dg Lie algebras in~$\mathrm{QCoh}(X)$.
\item In the setting of smooth algebraic varieties in characteristic zero, Calaque-C{\u{a}}ld{\u{a}}raru-Tu \cite{CCT} gave an adjunction between $X$-pointed, $\mathbb{K}$-linear formal moduli problems and Lie algebroids on~$X$.
\item Gaitsgory-Rozenblyum \cite{GR} greatly generalize the notion of Lie algebroid and formal moduli problems, to a theory internal to derived stacks (and phrased in terms of $(\infty, 2)$-categories).  In this setting they show that $X$-pointed formal moduli problems are equivalent to (their notion of) Lie algebroids on $X$.  
\end{enumerate}
Note the interesting variety of pointings and bases.

Costello's introduction of $\L8$ spaces was inspired by Kapranov's work: 
he wanted to rephrase complex manifolds in Lie-theoretic terms so as to reinterpret a $\sigma$-model mapping into such a manifold as a family of gauge theories living over that manifold.
His definition is, however, close in spirit to \cite{GR}, because an $\L8$ space lives between a manifold $X$ and its de Rham space $X_{dR}$.
Because this notion does not work relative to more sophisticated stacks, however, it naturally relates to the ordinary notion of Lie algebroid and does not require their generalization.

\section{Representations up to homotopy and vector bundles on $\L8$ spaces}

\subsubsection{}

The following notion of ``module over a Lie algebroid'' is the most relevant to our work. It is introduced in \cite{AbadCrainic} and allows for the construction of an adjoint representation of a Lie algebroid. Recall that for a vector bundle $E \to X$, we use $\cE$ to denote its sheaf of smooth sections.

\begin{definition}
A {\it representation up to homotopy} of a Lie algebroid $\rho \co L \to T_X$ is a $\ZZ$-graded vector bundle $E \to X$ and a dg $\calgd(L)$-module structure on the sheaf of free $\calgd(L)^\sharp$-modules $\calgd(L)^\sharp \otimes_{\cinf_X} \cE$. We denote this dg module by $\calgd(L,E)$ and call $E$ its {\em underlying vector bundle}. 
\end{definition}

That is,  $\calgd(L)^\sharp \otimes_{\cinf_X} \cE$ is equipped with the obvious $\calgd(L)^\sharp$ action by left multiplication. To specify a representation up to homotopy, one makes a choice of differential compatible with this graded $\calgd(L)^\sharp$-module structure, namely, a differential $D$ on $\calgd(L)^\sharp \otimes_{\cinf_X} \cE$ such that
\[
D_E(x s) = (d_L x) s + (-1)^{x} x( D_Es)
\]
for any section $x \in \calgd(L)$ and section $s \in \calgd(L)^\sharp \otimes_{\cinf_X} \cE$. This concept is also known as a super-representation in the work of Mehta and Gracia-Saz~\cite{GSM}. 

\begin{definition}
A {\it map of representations up to homotopy} is a map of dg $\calgd(L)$-modules $f \co \calgd(L,E) \to \calgd(L,E')$. We denote by $\RR {\rm ep}^\infty(L)$ this category of representations up to homotopy of~$L$.
\end{definition}

After Theorem \ref{thm:repvb}, we will discuss Arias Abad-Crainic's notion of weak equivalence of representations.

\begin{lemma}[\cite{AbadCrainic} Example 4.1] 
If $\calgd(L,E) \in \repinf(L)$, then there exists a unique representation up to homotopy $\calgd(L,E^\vee)$ such that
\begin{enumerate}
\item the underlying vector bundle of $\calgd(L,E^\vee)$ is the (graded) dual of $E$, and
\item for each $s \in \calgd(L,E)$ and $s' \in \calgd(L,E^\vee)$, the differential $D_{E^\vee}$ satisfies
\[
d_L (\op{ev}(s \otimes s') ) = \op{ev}(D_{E^\vee} (s) \otimes s') + (-1)^{\lvert s \rvert} \op{ev}(s \otimes D_E (s')),
\]
where 
\[
\op{ev}_E \co \calgd(L,E) \otimes \calgd(L,E^\vee) \to \calgd(L)
\]
is the $\calgd$-linear extension of the natural fiberwise evaluation pairing between sections of $E$ and~$E^\vee$.
\end{enumerate}
\end{lemma}

We call this representation up to homotopy $\calgd(L,E^\vee)$ the {\it dual representation up to homotopy}.

\subsubsection{}

Recall from section \ref{l8vb} the category of vector bundles over an $\L8$ space.  In particular, we will write $\op{VB}(\op{enh}(L))$ for the category of vector bundles over the $\L8$ space corresponding to a given Lie algebroid $\rho \co L~\to~T_X$. 

\begin{theorem}\label{thm:repvb}
There is a faithful functor $\op{enh}_{\op{mod}} \co \repinf(L)~\to~\op{VB}(\op{enh}(L))$.
\end{theorem}

\begin{proof}
Just as in the construction of $\op{enh}$, we will fix connections on the underlying vector bundles to make an explicit construction, but these choices are irrelevant up to isomorphism, by an argument identical to that given in the construction of $\op{enh}$. Thus, fix a connection on $L$ once and for all. The $\L8$ space $\op{enh} (L)$ is given by a pair~$(X,~\op{enh}(L))$. 

Let $\calgd(L,E)$ denote an $L$-representation up to homotopy with underlying graded vector bundle $E$ and let $\calgd(L,E^\vee)$ denote the dual representation with underlying bundle $E^\vee$. Fix a connection on $E$, which induces a connection on $E^\vee$ and also on $\Sym(L^\vee[-1]) \otimes E$.  By using the connections, we obtain an isomorphism of vector bundles
\[
J(\Sym(L^\vee[-1]) \otimes E)) \cong \csym(T^\vee_X \oplus L^\vee[-1]) \otimes E,
\]
as in the proof of Lemma \ref{lem:algebra}. This isomorphism induces an isomorphism of $\cinf_X$-modules
\[
dR_X(J(\calgd(L,E))^\sharp \cong \Omega^\sharp_X(\csym(T^\vee_X \oplus L^\vee[-1]) \otimes E),
\]
and  along this isomorphism we transfer the differential on $dR_X(J(\calgd(L,E))$ to a differential on the right hand side, making it a dg $\widehat{C}^* (\op{enh}(L))$-module. Let us denote this module by $\widehat{C}^* (\op{enh}(L), \widetilde{E})$.
We very nearly have a vector bundle over $\op{enh}(L)$: it remains to verify that this dg module is the sheaf of sections of some vector bundle. 

But this is simple. The dual representation up to homotopy  $E^\vee$ produces a dg $\widehat{C}^* (\op{enh}(L))$-module that we will denote $\widehat{C}^* (\op{enh}(L), \widetilde{E}^\vee)$. This module canonically provides the ``linear functions'' among the ``ring of functions'' on the total space~$\op{enh}(L\rtimes\widetilde{E})$:
\[
\widehat{C}^*(\op{enh}(L \rtimes \widetilde{E})) := \csym_{\widehat{C}^* (\op{enh}(L))} \widehat{C}^* (\op{enh}(L, \widetilde{E}^\vee)).
\] 
Hence our construction produces a vector bundle on $\op{enh}(L)$ whose underlying vector bundle on $X$ is just $\widetilde{E}~=~\Omega^\sharp_X(\Sym(T^\vee_X[-1])~\otimes~E )$.

A map of representations up to homotopy $f \co \calgd(L,E) \to \calgd(L,F)$ induces a $\widehat{C}^* (\op{enh}(L)^\sigma)$-linear map of sections $\op{enh}_{\op{mod}} (f) \co dR_X (J(\calgd(L,E))) \to dR_X(J(\calgd(L,F)))$.  Moreover, this association is faithful, since taking jets is injective.
\end{proof}

\begin{remark}\label{lalgdbeck}
The definition of a representation up to homotopy can seem less-than-obvious on first exposure, especially when formulated as an infinite sequence of higher homotopies. It is, however, essentially a module over the Lie algebroid, but viewed as an $\L8$ algebra. Equivalently, it is essentially a Beck module for the Lie algebroid. This perspective makes clear why we should have such a nice functor $\op{enh}_{\op{mod}}$: we simply apply $\op{enh}$ to the overcategory $\dglalgdbf(X)_{/L}$, which maps to the overcategory $\linfcatbf(X)_{/\op{enh}(L)}$. Compare to Remark~\ref{l8beck}.
\end{remark}

\begin{remark}
The preceding theorem is the analogue of the central result of \cite{Mehta}, where Mehta proves an equivalence between representations up to homotopy and Va{\u\i}ntrob's Lie algebroid modules.
\end{remark}

Both categories possess natural symmetric monoidal structures. In the case of representations up to homotopy, we use $-\otimes_{\calgd(L)}-$. For vector bundles on an $\L8$-space $(X,\fg)$, we essentially tensor as representations of the curved $\L8$ algebra $\fg$. In detail, if $\cV$ and $\cW$ are vector bundles, we tensor the  the underlying sheaves of sections over $\Omega^\sharp_X$ and then extend the action of $\fg$ in the standard way, i.e., a section $x$ of $\fg$ acts by $x \otimes \op{id}_{\cW} + \op{id}_{\cV} \otimes x$. The construction in the preceding proof manifestly intertwines these tensor products, and so we have the following.

\begin{lemma}
The functor $\op{enh}_{\op{mod}}$ is symmetric monoidal.
\end{lemma}

\subsubsection{}

These categories both possess notions of weak equivalence, and we will show that $\op{enh}_{\op{mod}}$ is a homotopy functor (i.e., respects weak equivalences). From section 4.2 of \cite{AbadCrainic}, we recall Arias Abad-Crainic's notion, which relies on the natural filtration 
\[
F^k \calgd(L,E) = \Gamma(X, \Lambda^{\geq k} L^\vee \otimes E)
\]
on $\calgd(L,E)$. This filtration is the Lie algebroid analog of the Hodge filtration $F^k \Omega^*_X = \Omega^{\geq k}_X$ on a de Rham complex.

\begin{definition}[Definition 4.9 of  \cite{AbadCrainic}]
A map $f \co \calgd(L,E) \to \calgd(L,E')$ of representations up to homotopy is a {\it weak equivalence} if the induced map 
\[
\tilde{f} \co \calgd(L,E)/F^1 \calgd(L,E) \to \calgd(L,E')/F^1 \calgd(L,E')
\]
is a quasi-isomorphism.
\end{definition}

\begin{remark}
Let $\cE$ denote the sheaf of smooth sections of the graded vector $E$. Note that $\calgd(L,E)/F^1 \calgd(L,E)$ is simply $\cE$ equipped with a $\cinf$-linear differential; from hereon we implicitly view $\cE$ as this {\em dg} vector bundle.
Further, we have an isomorphism
\[
\Gr \calgd(L,E) \cong (\Gr \calgd(L)) \otimes_{\cinf} \calgd(L,E)/ F^1 \calgd(L,E),
\]
so we see that $\Gr f$ is simply $\id_{\Gr \calgd(L)} \otimes \tilde{f}$. Hence, $\Gr f$ is a quasi-isomorphism if and only if $\tilde{f}$ is a quasi-isomorphism.
That is, a weak equivalence is simply a filtered quasi-isomorphism.
\end{remark}

\begin{prop}
The functor $\op{enh}_{\op{mod}}$ preserves weak equivalences. 
Thus, it induces a functor at the level of derived ({\it aka} homotopy) categories.
\end{prop}

\begin{proof}
Let $f \co \calgd(L,E) \to \calgd(L,E')$ be a weak equivalence of representations up to homotopy. By definition, $\tilde{f} \co \cE \to \cE'$ is a quasi-isomorphism. We need to show that the induced map of vector bundles on $\op{enh} \; L$ is a quasi-isomorphism.  Recall---see the proof of Theorem \ref{thm:repvb}---that $\op{enh}_{\op{mod}} (C^\ast (L,E))$ has underlying graded vector bundle $\widetilde{E} := \Omega^\sharp_X (\Sym(T^\vee_X[-1]) \otimes E )$ on the manifold $X$, and similarly for $C^\ast (L,E')$. Now the map of vector bundles 
\[
\op{enh}_{\op{mod}} (f) \co \op{enh}_{\op{mod}} (C^\ast(L,E)) \to  \op{enh}_{\op{mod}} (C^\ast(L,E'))
 \]
induces a map of sheaves of sections of dg vector bundles
\[
\left (\op{enh}_{\op{mod}} (f) \right )_\fib \co \widetilde{\cE} \to \widetilde{\cE'}
\]
on $X$, where $\widetilde{\cE}$ and $\widetilde{\cE'}$ are equipped with the $\cinf$-linear differentials described just before Definition \ref{vbwe}.
We need to show that this map is a quasi-isomorphism but by construction
\[
\left (\op{enh}_{\op{mod}} (f) \right )_\fib = \id_{\Omega^\sharp_X} \otimes \id_{\Sym (\cT^\vee_X[-1])} \otimes \tilde{f},
\]
so the proposition follows. 
\end{proof}

\subsection{The adjoint complex(es) and the deformation complex}

In the setting of Lie algebroids, there are two cochain complexes that play the role of the ``tangent complex'' of a Lie algebroid. On the one hand, there is the {\it deformation complex} of Crainic-Moerdijk \cite{CrainicMoerdijk}; and on the other, there is the {\it adjoint complex} of Arias Abad-Crainic \cite{AbadCrainic}. In the end, these constructions are isomorphic (after a shift in degree), but their definitions have rather different flavors. The deformation complex has a more intrinsic flavor---it is an obvious variant of the Hochschild complex for associative algebras and the deformation complex of Lie algebras---and it plays the starring role in our constructions below, so we focus on it here.

Our main goal in this section is to show that $\op{enh}_{\op{mod}}$ sends the deformation complex of a Lie algebroid $L$ (equivalently, its shifted adjoint complex) to the tangent bundle of the $\L8$ space~$\op{enh}(L)$.

\subsubsection{The deformation complex}

Let $E \to X$ be a graded vector bundle. 
Recall that a {\it derivation} of its sections $\cE(X)$ is an $\RR$-linear endomorphism $D$ such that there is a vector field $\sigma_D$ with the property that
\[
D(fe) = \sigma_D(f) e + f D(e)
\]
for every section $e$ and every smooth function $f$.
Analogously, an {\it $n$-multiderivation} of its sections $\cE(X)$ is a graded-antisymmetric, $\RR$-multilinear map
\[
D \co \underbrace{\cE(X) \otimes \cdots \otimes \cE(X)}_{n+1 \text{ times}} \to \cE(X)
\]
that is a derivation in each entry separately. (Note the potential for terminological confusion here: the map $D$ has some cohomological degree and it has ``degree $n$'' as a multiderivation, which just depends on the number of inputs.) Thus, there is a {\it symbol map} $\sigma_D \co \cE(X)^{\otimes n} \to \cT_X(X)$ defined by
\[
D(e_0,\ldots,e_{n-1},fe_n) = f D(e_0,\ldots,e_n) + \sigma_D(e_1,\ldots,e_{n-1})(f)e_n,
\]
for any smooth function $f \in \cinf(X)$ and sections $e_0,\ldots,e_n \in \cE(X)$. When the rank of $E$ is greater than one, every symbol map $\sigma_D$ is automatically $\cinf$-linear. We restrict to multiderivations with $\cinf$-linear symbols when $\op{rk}(E) = 1$. Observe that $0$-derivations are simply derivations of $\cE(X)$ and that $-1$-derivations are simply sections of~$\cE(X)$.
(We attempt to motivate this definition of multiderivation in Section \ref{sec:whymultider} below.)

The first result of Crainic-Moerdijk identifies this algebraic construction with a geometric object.

\begin{lemma}[Lemma 1 of \cite{CrainicMoerdijk}]
The graded vector space $\cD er^n(E)$ of $n$-multiderivations is equal to the sections of a graded vector bundle $\op{Der}^n(E)$. This vector bundle sits in a short exact sequence
\[
\Lambda^{n+1} E^\vee \otimes E \to \op{Der}^n(E) \to \Lambda^n E^\vee \otimes T_X.
\]
A choice of connection on $E$ induces a splitting of this exact sequence.
\end{lemma}

The projection map arises from taking the symbol of an $n$-multiderivation. A connection $\nabla$ allows one to split the inclusion map via
\[
L_D( e_0,\ldots,e_n) = D(e_0,\ldots,e_n) + (-1)^n \sum_i (-1)^{i+1} \nabla_{\sigma_E(\ldots,\hat{e}_i,\ldots)}(e_i),
\]
which is antisymmetric and $\cinf$-multilinear.
Note that the component $\Lambda^{n+1} E^\vee \otimes E$ encodes $\cinf$-linear multiderivations of $\cE$.
Hence, the splitting allows us to recognize that the ``entrywise derivation'' condition simply adds the component $\Lambda^n E^\vee \otimes T_X$, 
which encodes the symbols.

Let $\op{Der}^*(E)$ denote the graded vector bundle $\bigoplus_{n \geq -1} \op{Der}^n(E)[-n]$. Its sheaf of sections $\cD er^*(E)$ has a natural (graded) Lie algebra structure by the standard ``commutator bracket'' for multilinear operators. This statement is the second result in Crainic-Moerdijk.

\begin{lemma}[Prop. 1 of \cite{CrainicMoerdijk}]
Given $D_1 \in \cD er^p(E)$ and $D_2 \in \cD er^q(E)$, the {\it circle product} is
\[
D_2 \circ D_1 (e_0,\ldots,e_{p+q}) = \sum_\pi (-1)^\pi D_2(D_1(e_{\pi(0)},\ldots,e_{\pi(p)}), e_{\pi(p+1)},\ldots,e_{\pi(p+q)}),
\]
where $\pi$ runs over all $(p+1,q)$-shuffles. The {\it Gerstenhaber bracket}
\[
[D_1,D_2] = (-1)^{pq} D_1 \circ D_2 - D_2 \circ D_1
\]
makes $\cD er^*(E)$ into a Lie algebra.
\end{lemma}

This Lie algebra encodes information of great interest to us.
For instance, if $E$ is concentrated in degree zero, a Lie algebroid structure on $E$ is precisely an element $m \in \cD er^1(E)$ such that $[m,m] = 0$; the $m$ provides a Lie bracket on $\cE$. 
In other words, Maurer-Cartan elements of the dg Lie algebra $(\cD er^*(E), [-,-])$---equipped with the zero differential---are deformations of $E$ from a trivial Lie algebroid to a (possibly) non-trivial Lie algebroid. If $E$ is graded, then a dg Lie algebroid structure on $E$ is encoded by a Maurer-Cartan element of the form $m_L=m_{(0)} + m_{(1)}$, where $m_{(0)}$ is a 0-derivation providing the differential and where $m_{(1)}$ is a 1-derivation providing the Lie bracket. (A natural definition of $\L8$ algebroid structure on $E$ would be a Maurer-Cartan element of $(\cD er^*(E), [-,-])$, although we will not pursue that notion here.) We are thus led to the following definition.

\begin{definition}
The {\it deformation complex} of a dg Lie algebroid $\rho \co L \to T_X$ is the graded Lie algebra of multiderivations $\cD er^*(L^\sharp)$ equipped with the differential $d = [m_L,-]$, where $m_L \in \cD er^1(L^\sharp)$ satisfies $[m_L,m_L] = 0$ and encodes both the differential and the Lie bracket on $\cL^\sharp$. We denote this dg Lie algebra by~$\cD ef(L)$.
\end{definition}

\begin{remark}
Crainic and Moerdijk use a shift $\cD ef(L)[-1]$ of this dg Lie algebra (cf. section 2.4 of \cite{CrainicMoerdijk}), presumably to match the classical convention that for any ordinary Lie algebra $\fg$, a degree 2 cocycle of $C^*(\fg,\fg)$ encodes a first order deformation of the Lie bracket. We prefer to work with $\cD ef(L)$ as we want to have a dg Lie algebra describing deformations, rather than a shifted dg Lie algebra; in other words, we prefer our Maurer-Cartan elements to live in degree 1. An alternative explanation for our preference is that for the commutative dg algebra $C^*(\fg)$, the tangent complex is $C^*(\fg,\fg[1])$, using standard conventions of homological algebra.
\end{remark}

A choice of connection on $L$ allows one to identify $\cD ef(L)[-1]$ with a representation up to homotopy of $L$ with underlying vector bundle $T_X[-1] \oplus L^\sharp$: the splitting of $\op{Der}^*(L)$ induced by the connection gives an isomorphism of graded sheaves
\[
\cD er^*(L^\sharp)[-1] \cong \cS {\rm ym}(L^\vee[-1]) \otimes_{\cinf} (\cT_X[-1] \oplus \cL^\sharp),
\]
and we simply the transfer the differential of $\cD ef(L)[-1]$ along this isomorphism. This procedure is explained in the proof of Theorem 3.11 of \cite{AbadCrainic}, and so provides an alternative definition of the adjoint complex.

\begin{definition}
Given a choice of connection $\nabla$ on the underlying graded vector bundle $L^\sharp$ of a dg Lie algebroid $L$, the {\it adjoint complex} $\op{adj}(L,\nabla)$ of $L$ associated to $\nabla$ is the representation up to homotopy induced by the (shifted) deformation complex~$\cD ef(L)[-1]$.
\end{definition}

It is thus manifest that all adjoint complexes are naturally isomorphic, and not just quasi-isomorphic.

\subsubsection{An explanation for this definition}
\label{sec:whymultider}

The reader might wonder how one would go about inventing this definition of $\cD ef(L)$.
The subtle condition is that an $n$-multiderivation $D$ is a derivation in each entry,
so we focus on explaining where it comes from.

First, consider two natural variants of the construction of $\cD ef$ that are likely more familiar.
If one dropped this derivation condition and simply worked $\RR$-linearly, 
then the remaining pieces of the construction encode the dg Lie algebra of graded derivations of the commutative dg algebra $\Sym_\RR(\cL^*[-1])$.
Indeed, the $n$-multiderivations are a subspace of this big $\RR$-linear construction.
For instance, the Gerstenhaber bracket is just the bracket of derivations on that dg $\RR$-algebra.
On the other hand, if we ignored the Lie algebroid structure on $L$ but worked $\cinf$-linearly, 
then the underlying graded algebra of $\calgd(L)$ is $\Sym_{C^\infty}(\Gamma(L^\vee)[-1])$ 
and the graded derivations are $\cinf$-linear.
Hence the dg Lie algebra of $\cinf$-linear derivations is $\Sym_{C^\infty}(\Gamma(L^\vee)[-1]) \otimes_\cinf \Gamma(L)$.
Note that this object does provide a summand of~$\cD ef(L)$.

The subtle condition on multiderivations tries to fit between these two variants,
just as a Lie algebroid tries to fit between an $\RR$-linear Lie algebra and a $\cinf$-linear Lie algebra. 
Indeed, consider the underlying graded-commutative $\RR$-algebra of $\calgd{L}$: 
\[
\mathsf{C}^\sharp(L) = \Gamma(X,\Sym(L^\vee[-1]).
\]
The graded derivations of this algebra over $\RR$ (i.e., {\em not} over $\cinf(X)$) naturally form a graded Lie algebra, via the commutator bracket.
As shown in section 2.5 of \cite{CrainicMoerdijk}, every $n$-multiderivation $D$ of $L$ determines such a graded derivation of $\mathsf{C}^\sharp(L)$, by a formula analogous to the Lie derivative.
(See equations (6)-(9) therein.)
In fact, they prove the following.

\begin{lemma}
\label{cf der with der}
This Lie derivative-type construction determines an isomorphism of graded Lie algebras from $\cD er^*(L^\sharp)$ to the graded $\RR$-linear derivations of $\mathsf{C}^\sharp(L)$.
\end{lemma}

The idea of the construction is easy to see in the simplest case.
Let $D$ be a multiderivation of degree $0$. 
It determines a linear endomorphism $L_D$ of $\mathsf{C}^\sharp(L)$ by 
\[
L_D(\alpha)(\ell_0,\ldots,\ell_n) = L_{\sigma_D}(\alpha(\ell_0,\ldots,\ell_n)) - \sum_{i=0}^n \alpha(\ell_0,\ldots, D\ell_i, \ldots,\ell_n).
\]
Note that in particular that for a single section
\[
L_D(a)(\ell) = L_{\sigma_D}(a(\ell)) - a(D\ell)
\]
and so
\begin{align*}
L_D(fa)(\ell) &= L_{\sigma_D}(fa(\ell)) - fa(D\ell) \\
&= \sigma_D(f) a(\ell) + f(L_{\sigma_D}(a(\ell)) - a(D\ell))
\end{align*}
The higher degree multiderivations simply involve many extra factors.

This approach may seem strange, as it does not to depend on the Lie algebroid structure but only on the underlying vector bundle.
But this is precisely what we want if we know we will be working with {\em some} Lie algebroid structure on $L$, 
but do not want to fix it ahead of time. 
(That is, we want a construction that recovers the correct derivations but is uniform in all Lie algebroid structures.)
Then it is natural to ask that we have an $\RR$-linear derivation that is entrywise a derivation,
while not specifying the form of that entrywise derivation.
This is what the $n$-multiderivation condition does.
The differential $[m_L,-]$ is used to enforce
compatibility with a choice of Lie algebroid structure.

\subsubsection{$\cD ef$ and the tangent bundle of an $\L8$ space}

We now want to relate the deformation ({\it aka} adjoint) complex of a Lie algebroid $L$ to the tangent bundle of its $\L8$ space $\op{enh}(L)$. Recall that for an $\L8$ space $(X,\fg)$, the {\it tangent bundle} is the vector bundle $\fg[1]$ equipped with the adjoint action of $\fg$. Hence, the sections of this tangent bundle are the Chevalley-Eilenberg complex~$\widehat{C}^*(\fg,\fg[1])$.

\begin{prop}\label{prop:tangent}
For each dg Lie algebroid $L$ over the manifold $X$, there is an isomorphism of dg Lie algebras between 
\begin{itemize}
\item the tangent bundle $T_{\op{enh}(L)}$ of the $\L8$ space $\op{enh}(L)$ associated to $L$ and
\item the vector bundle $\op{enh}_{\op{mod}}(\op{adj}(L,\nabla)[1])$ on $\op{enh}(L)$ associated to an adjoint complex of~$L$.
\end{itemize}
\end{prop}

To be clear here, the Lie structure on $T_{\op{enh}(L)}$ is by the bracket as vector fields;
in other words, we view elements as derivations of the commutative dg algebra $\widehat{C}^*(\op{enh}(L))$ over $\Omega^*_X$ and work with commutators of derivations. On the other hand, the Lie structure on $\op{enh}_{\op{mod}}(\op{adj}(L,\nabla)[1])$ is transferred from the Lie structure on the adjoint complex itself because $\op{enh}_{\op{mod}}$ is symmetric monoidal.

\begin{proof}
We begin by verifying the claim in a particularly simple case: the dg Lie algebroid $L$ has only zero brackets. In other words, it is simply a graded vector bundle that we will denote $L_0$, in order to emphasize the triviality of the brackets. Fix a connection on $L_0$, and so on $L_0^\vee$, and a connection on $T_X$, and so on~$T^\vee_X$.

In this case, several constructions simplify substantially. For example, the Chevalley-Eilenberg complex $\calgd(L_0)$ is simply $\cS{\rm ym}(L^\vee_0[-1])$ and the differential is zero. Hence, using our connections, we obtain an $\L8$ space $\op{enh}(L_0)$ via
\begin{align*}
dR_X(J({\rm Sym}(L^\vee_0[-1]))) &\cong \Omega_X^*(\Sym(T^\vee_X \oplus L^\vee_0[-1])) \\
&\cong \widehat{C}^*(\op{enh}(L_0)).
\end{align*}
As another simplification, $\cD ef(L_0)$ is a differential graded Lie algebra with trivial differential. Fixing a connection on $L_0$, we find that the associated adjoint complex is $\cS{\rm ym}(L^\vee_0[-1]) \otimes_{\cinf} (\cT_X[-1] \oplus \cL_0)$. Using our connections, we obtain an isomorphism $\cJ \cong \cS{\rm ym}(T^\vee_X)$ and thus an isomorphism
\[
\cJ(\op{adj}(L_0,\nabla)) \cong \cS{\rm ym}(T^\vee_X \oplus L^\vee_0[-1]) \otimes_{\cinf} (\cT_X[-1] \oplus \cL_0).
\]
Taking the de Rham complex for the Grothendieck connection on jets, we obtain an isomorphism 
\begin{align*}
dR_X(J(\op{adj}(L_0,\nabla))) &\cong \Omega_X^*(\Sym(T^\vee_X \oplus L^\vee_0[-1]) \otimes (T_X[-1] \oplus L_0)) \\
&\cong \widehat{C}^*(\op{enh}(L_0),\op{enh}(L_0)).
\end{align*}
Hence our construction $\op{enh}_{\op{mod}}$ sends the shifted deformation complex $\cD ef(L_0)[-1]$ to
\[
\widehat{C}^*(\op{enh}(L_0),\op{enh}(L_0)) \cong \widehat{C}^*(\op{enh}(L_0),\op{enh}(L_0)[1])[-1]  = T_{\op{enh}(L_0)}[-1],
\] 
the shifted tangent bundle of the $\L8$ space~$\op{enh}(L_0)$.

This argument simply identifies the underlying bundles and connections.
We now explain why these isomorphisms actually respect the Lie brackets.

First, we recall an important fact as background.
As noted in Remark \ref{rmk: jet of diff op}, 
a differential operator $P \co \cE \to \cF$ between sections of vector bundles determines a $\cD_X$-linear map $J(P) \co \cJ(E) \to \cJ(F)$;
conversely, every such $\cD_X$-linear map determines a differential operator.
Hence for a fixed graded vector bundle $E$, there is a natural isomorphism
\[
J: {\rm Diff}(\cE,\cE) \xto{\cong} \Hom_{\cD_X}(\cJ(E),\cJ(E))
\]
of graded Lie algebras, where we use the commutator bracket on both sides.

As explained in section \ref{sec:whymultider}, in the guise of Lemma \ref{cf der with der},
the multiderivations and their commutator are precisely the graded Lie algebra of graded derivations, over $\RR$, of $\calgd(L_0) = \mathsf{C}^\sharp(L_0)$.
Such graded derivations are, among other things, differential operators on $\calgd(L_0)$.
In particular, the graded Lie algebra of multiderivations is simply a sub-graded Lie algebra of ${\rm Diff}(\calgd(L_0),\calgd(L_0))$ and hence of $\Hom_{\cD_X}(\cJ(\calgd(L_0)),\cJ(\calgd(L_0)))$.
For a multiderivation $D$, let $J(D)$ denote the corresponding $\cD_X$-linear endomorphism of~$J(\calgd(L_0))$.

Now consider 
\[
Der^*_{\cD_X}(\cJ(\calgd(L_0))) \subset \Hom_{\cD_X}(\cJ(\calgd(L_0)),\cJ(\calgd(L_0))),
\]
the sub-graded Lie algebra of graded derivations of $\cJ(\calgd(L_0))$ as a graded-commutative algebra in the category of $\cD_X$-modules (using tensor over $\cinf(X)$ as the symmetric monoidal structure).
This operator $J(D)$ lives in this sub-graded Lie algebra, as it is a graded derivation of $\calgd(L_0)$,
so that we have a map of graded Lie algebras
\[
J: \cD er^*(L_0) \to Der^*_{\cD_X}(\cJ(\calgd(L_0)))
\]
by restriction.

This map is, in fact, an isomorphism, as follows.
Note that every graded derivation $\delta \in Der^*_{\cD_X}(\cJ(\calgd(L_0)))$ commutes with the flat connection on $J(\calgd(L_0))$, as it is $\cD_X$-linear. 
Hence $\delta$ preserves the sheaf of horizontal sections of $J(\calgd(L_0))$.
But these horizontal sections are precisely $\calgd(L_0)$ by Proposition \ref{prop:grconn}. 
Hence $\delta$ determines a multiderivation $D_\delta$,
by restricting to the horizontal sections of~$J(\calgd(L_0))$.
This construction is inverse to the map $J$ by direct inspection.

So far we have spoken of maps between global sections, 
but it becomes convenient now to talk at the level of  sheaves.
Recall that for any $\cD_X$-modules $\cM,\cN$,
there is a natural map
\[
\shom_{\cD_X}(\cM,\cN) \hookrightarrow dR(\shom_{\cinf}(\cM,\cN)),
\]
because a ${\cD_X}$-linear map is a $\cinf$-linear map and the the canonical ${\cD_X}$-module structure on $\shom_{\cinf}(\cM,\cN)$ picks out $\cD_X$-linear maps as horizontal sections.
By construction, this map is a quasi-isomorphism.
Hence there is a natural map
\[
Der^*_{\cD_X}(\cJ(\calgd(L_0))) \hookrightarrow dR(Der^*_{\cinf_X}(\cJ(\calgd(L_0)))
\]
by restricting to the derivations inside all maps.

Stringing all of our identifications together,
we obtain a map of sheaves
\[
\cD er^*(L_0) \xto{\cong} Der^*_\RR(\calgd(L_0)) \xto{\cong} Der^*_{\cD_X}(\cJ(\calgd(L_0))) \hookrightarrow dR(Der^*_{\cinf_X}(\cJ(\calgd(L_0))).
\]
In words, it says that every multiderivation determines a horizontal section of the de Rham complex of derivations of jets.
This composite is a quasi-isomorphism.
As taking the de Rham complex of left ${\cD_X}$-modules is symmetric monoidal,
the rightmost complex is isomorphic to 
\[
Der^*_{\Omega^*_X}(dR(J(\calgd(L_0)))) \cong Der^*_{\Omega^*_X}(\widehat{C}^*(\op{enh}(L_0)))
\]
which is an equivalent description of~$T_{\op{enh}(L_0)}$.

This complicated composite map is about how operations on smooth sections determine corresponding operations on their jets.
Since the map intertwines the actions as derivations, it manifestly enhances the earlier identifications involving the adjoint complex, which only involved the vector bundle structures.

\begin{remark}
We remark that another approach is to write out explicitly the Lie brackets, done most concretely by working in local coordinates with a choice of frame of $L_0$. For the authors at least, this approach did not illuminate why they agree.
\end{remark}

Having established the proposition for a trivial Lie algebroid, we turn to the general case.
It follows, in fact, as a deformation of the trivial case just explained. If the dg Lie algebroid $L$ corresponds to a Maurer-Cartan element $m_L$ in $\cD ef(L_0)$, then $j_\infty(m_L)$ provides a Maurer-Cartan element of $dR_X(J(\op{adj}(L_0,\nabla)[1]))$ and hence of $\widehat{C}^*(\op{enh}(L_0),\op{enh}(L_0)[1])$. Since the latter dgla controls deformations of $\op{enh}(L_0)$, we thus obtain a new $\L8$ space;  we want to identify it with $\op{enh}(L)$. 
Let us use the connection on $L_0$ to produce $\op{enh}\;L$. In that case, we are using the same underlying isomorphism of $\cinf_X$-modules
\[
J(\csym(L^\vee[-1])) \cong \csym(T^\vee_X \oplus L^\vee[-1])
\]
for both $L_0$ and $L$. Thus, there is a canonical isomorphism of $\cinf$-modules
\[
dR_X(J(\calgd(L_0)))^\sharp \cong dR_X(J(\calgd(L)))^\sharp \cong \Omega_X^\sharp(\csym(T^\vee_X \oplus L^\vee[-1])).
\]
The difference between the differentials on the first two spaces is precisely $j_\infty(d_L)$, the differential on $\calgd(L)$ that arises from the nontrivial dg Lie algebroid structure. Under the identification between a dg Lie algebra structure (i.e., differential and bracket) and the differential on its Chevalley-Eilenberg cochains, the Maurer-Cartan element $m_L$ identifies with $d_L$ and so $j_\infty(m_L)$ identifies with $j_\infty(d_L)$.  Hence, the deformation of $\op{enh}(L_0)$ by $j_\infty(m_L)$, transferred along the splitting isomorphism, is precisely~$\op{enh}(L)$. 
\end{proof}

\begin{remark}
This proof is independent of Theorem \ref{thm:construction} and actually provides an alternative proof. One starts by checking directly the case of a trivial Lie algebroid, which is a straightforward application of results about $\infty$-jets, and then applies the deformation to obtain the general case.
\end{remark}

\begin{remark}\label{l8algebroid}
In the proof above, we could have worked with an arbitrary Maurer-Cartan element $\cD ef(L_0)$, which encodes an $\L8$ algebroid, rather than an element encoding a dg Lie algebroid. This notion of $\L8$ algebroid is well-behaved, but we feel that it would add an unnecessary layer of complexity to this paper to develop the full categorical aspects of $\L8$ algebroids necessary for articulating our main results at this level of generality. Several such aspects are described in \cite{OhPark} under the name of {\it strong homotopy Lie algebroid} or in \cite{Kjeseth} as {\it homotopy Lie-Rinehart pairs}. More recent appearences of this notion can be found in \cite{CCT, Poncin, SSS}.  We think it would be useful and interesting to see the various flavors of this formalism unified and expanded.
\end{remark}

\begin{remark}
An alternative generalization of representations up to homotopy has been put forth by Vitagliano \cite{Luca}.  Further, his $\mathit{LR}_\infty$ modules also generalize actions of $\L8$ algebras on dg manifolds.  It could prove useful to relate his notions to our constructions in $\L8$ spaces.
\end{remark}

\subsection{The Weil complex as a de Rham complex}

Under our correspondence between representations up to homotopy of $L$ and vector bundles on $\op{enh}(L)$, the Weil algebra of $L$ goes to the de Rham complex of $\op{enh}(L)$, as we now explain. We thus obtain new perspectives on this de Rham complex in light of the prior work on the Weil algebra. For instance,  Arias Abad and Crainic \cite{AbadCrainic} show this Weil complex includes the BRST complex of Kalkman \cite{Kalkman}. In other work, they show it also has a natural role in studying the cohomology of classifying spaces~\cite{ACWeil,ACBott}.

\subsubsection{The Weil algebra in Lie theory}

Recall that the {\em Weil algebra} $W(\fg)$ of an ordinary Lie algebra $\fg$ is a commutative dg algebra whose underlying graded algebra is $\Sym(\fg^\vee[-1] \oplus \fg^\vee[-2])$ and whose differential has the following form. Note that $d_{W(\fg)}$ is a derivation and thus determined by its behavior on the subspace of generators $\fg^\vee[-1] \oplus \fg^\vee[-2]$. The degree 2 component of $W(\fg)$ consists of $\Lambda^2 \fg^\vee \oplus \fg^\vee$, and on the degree one component the differential $d_{W(\fg)}$ breaks up as a sum $d_{CE} + d_{dR}$, where $d_{CE}$ is the differential in the Chevalley-Eilenberg cochain complex $C^*(\fg)$ and where $d_{dR}$ is simply the identity. The degree 3 component of $W(\fg)$ consists of $\Lambda^3 \fg^\vee \oplus \Lambda^2 \fg^\vee \otimes \fg^\vee \oplus \Sym^2 \fg^\vee$, and on the linear piece $\fg^\vee$ of the degree 2 component, the differential $d_{W(\fg)}$ breaks up as a sum of three terms of which only the term $\fg^\vee \to  \Lambda^2 \fg^\vee \otimes \fg^\vee$ is nontrivial. It is given by the differential of $C^*(\fg, \fg^\vee)$, and so is determined by the coadjoint action. The Weil complex has $H^k W(\fg) = 0$ for $k \neq 0$ and $H^0 W(\fg)~=~\CC$.

There is a  more succinct and conceptual way to obtain the Weil algebra: it is the de Rham complex of the commutative dg algebra $C^*(\fg)$. We quickly outline this construction. First, identify $C^*(\fg,\fg^\vee[-1])$ as the K\"ahler differentials of the Chevalley-Eilenberg cochains, which thus possesses a universal derivation map $d_{dR} \co C^*(\fg) \to C^*(\fg, \fg^\vee)$. Extending this de Rham differential to higher de Rham forms in the usual fashion, we obtain a double complex
\[
C^*(\fg) \xto{d_{dR}} C^*(\fg, \fg^\vee[-1]) \xto{d_{dR}} C^*(\fg, \Lambda^2(\fg^\vee[-1])) \to \cdots
\]
whose totalization we call the de Rham complex of $C^*(\fg)$. It is manifest from this construction that it is the Weil algebra. 

We can now explain why the Weil complex has trivial cohomology, except in degree 0. In the double complex constructed in the previous paragraph, the vertical differential in column $p$ is simply the differential on $C^*(\fg, \Sym^p(\fg)[-p])$. If we ignore the vertical differential and use just the horizontal differential, the total complex is the symmetric algebra on the two-term complex $\fg[-1] \xto{\id} \fg[-2]$. (This claim is just a version of the fact that the de Rham differential on polynomials sends a generator $x$ to the 1-form $dx$.)  Hence, using the spectral sequence of the filtration by the internal, or vertical, degree of this de Rham complex for $C^*(\fg)$, we see that the first page is trivial except for $p = 0 =q$ by the Poincar\'e lemma for polynomials.

\subsubsection{The Weil algebra for Lie algebroids}

In section 5 of \cite{AbadCrainic}, Arias Abad and Crainic  define a Weil algebra $\cW (L, \nabla)$ associated to a Lie algebroid $L$ with a choice of connection $\nabla$ on the underlying vector bundle. This connection is used to define a coadjoint representation $((\op{adj}\,L)^\vee , \nabla)$, but every choice of connection produces an isomorphic Weil algebra.

Their Weil algebra is the double complex with 
\[
\cW (L, \nabla)^{p,q} = \bigoplus_{k} \Gamma(X,  \Lambda^{p-k} (L^\vee) \otimes \Lambda^{q-k} (T^\vee_X) \otimes \Sym^k (L^\vee))
\]
and with differential a sum of a horizontal and vertical component.
In other words, the underlying graded algebra of the Weil algebra is generated over smooth functions $\cinf(X)$ by a copy of $\Gamma(X,L^\vee)$ equipped with bidegree (1,0), a copy of $\Omega^1(X) = \Gamma(X,T^*_X)$ equipped with bidegree (0,1), and a copy of $\Gamma(X,L^\vee)$ equipped with bidegree (1,1). (To compare with our discussion above, note that Arias Abad and Crainic swap the vertical and horizontal directions relative to our construction. We will adhere to their conventions here to simplify comparison.) 
To specify the differential, it is thus enough to say how it acts on these generators.

We describe the vertical differential first. The column $\cW^{0,*}(L,\nabla)$ is simply a copy of the de Rham complex $\Omega^*(X)$, and hence the vertical differential on this column is simply the exterior derivative. 
The column $\cW^{1,*}(L,\nabla)$ has underlying graded vector space $\Omega^\sharp(X, L^\vee) \oplus \Omega^\sharp(X, L^\vee)[-1]$, but the vertical differential is a little complicated: if $(a,b)$ is an element of the direct sum, then
\[
d^{ver}((a,b)) = (\nabla a - R_\nabla b, \nabla b + a),
\]
where $\nabla$ here denotes the chosen connection (not usually flat!) on $L^\vee$ and $R_\nabla$ denotes the associated curvature.
Note that the vertical differential on this Weil algebra is thus a  natural Lie algebroid analog of what we called the horizontal differential above, which arose from the identity map $\fg \to \fg$. Arias Abad and Crainic express this perspective using a ``double" construction (see example 3.8 in~\cite{AbadCrainic}).

The horizontal differential is a bit simpler to describe.
The row $\cW^{*,0}(L,\nabla)$ is simply a copy of $\calgd(L)$.
The row $\cW^{*,1}(L,\nabla)$ is the coadjoint representation $((\op{adj}\,L)^\vee , \nabla)$ with the differential ``conjugated,'' in the terminology of Arias Abad-Crainic, which simply means to modify the sign of the differential in the usual way due to shifting the complex up by one degree.
A direct computation verifies that the vertical differential mapping the zeroth row to the first row is a derivation from the commutative dg algebra $\calgd(L)$ to the coadjoint representation; in other words, it is a kind of exterior derivative from functions to 1-forms.

The following result should thus come as no surprise.

\begin{prop}\label{prop:weil}
The de Rham complex of $\op{enh}(L)$ is the image under $\op{enh}_{\op{mod}}$ of the (totalization of the) Weil complex of~$L$.
\end{prop}

\begin{proof}
Recall that the functor $\op{enh}_{\op{mod}}$ is symmetric monoidal and also the construction of the dual of a representation up to homotopy. Combining these, we find that  
\[
\op{enh}_{\op{mod}}(\Sym^q(\op{adj}^\vee(L,\nabla))[-p]) \cong \widehat{C}^*(\op{enh}(L), \Sym^p(T^\vee_{\op{enh}(L)})[-p]).
\]
It remains to verify that $\op{enh}_{\op{mod}}$ sends the ``vertical differential'' $d^{ver}$, in the sense of Arias Abad-Crainic, to the de Rham differential $d_{dR}$ for~$\widehat{C}^*(\op{enh}(L))$.

This vertical differential is a derivation, as is the de Rham differential of $\widehat{C}^*(\op{enh}(L))$, so it suffices to understand how they behave on the generators. In other words, we only need to show that 
\[
\begin{array}{cccc}
\op{enh}_{\op{mod}}(d^{ver}) \co & \op{enh}_{\op{mod}}(\calgd(L)) & \to & \op{enh}_{\op{mod}}(\op{adj}^\vee(L)) \\
& \parallel & & \parallel\\
& \widehat{C}^*(\op{enh}(L)) & & \widehat{C}^*(\op{enh}(L), T^\vee_{\op{enh}(L)})
\end{array}
\]
is the universal derivation on $\widehat{C}^*(\op{enh}(L))$. To show this, we use the characterizing property that 
\[
L_{\cX} f = \langle \cX, d_{dR} f \rangle
\]
for $f$ an element of $\widehat{C}^*(\op{enh}(L))$, $\cX$ a derivation on $\widehat{C}^*(\op{enh}(L))$ (i.e., a section of the tangent bundle), $L_\cX$ the Lie derivative, and $\langle -,-\rangle$ the evaluation pairing between vector fields and one-forms. On the level of representations up to homotopy, the vertical differential satisfies the analogous relation 
\[
L_{\cX} f = \langle \cX, d^{ver} f \rangle
\]
for $f$ an element of $\calgd(L)$, $\cX$ an element of the adjoint complex, $L_\cX$ the Lie derivative, and $\langle -,-\rangle$ the evaluation pairing between the adjoint and coadjoint complexes. (Verification is easiest using the invariant---i.e., connection-independent---description of the (co)adjoint complexes in Section 3.2 and Example 4.7 of \cite{AbadCrainic}.) By Proposition \ref{prop:tangent}, we see that $\op{enh}_{\op{mod}}$ intertwines the Lie derivative at the level of the adjoint complex with the Lie derivative at the $\L8$-space level. Hence, $\op{enh}_{\op{mod}}(d^{ver})~=~d_{dR}$.

Alternatively, since the Weil complexes are isomorphic (not just quasi-isomorphic) for any choice of connection on $L$, we can check locally using a convenient connection. Let $U$ be a coordinatized open subset of $X$ on which $L$ is trivializable, and fix a frame on $L$ and let $\nabla$ be the associated flat connection on $L$. In that case, the formulas in Remark 5.2 of \cite{AbadCrainic} simplify tremendously. In particular, the vertical differential from $\calgd(L)$ to the coadjoint complex is, in degree zero, the exterior derivative from $\cinf(X)$ to $\Omega^1(X)$ and, in degree one, the identity from $\mathsf{C}^1(L)$ to the copy of $\Gamma(U,L^\vee)$ in the degree 1 component of the coadjoint complex. After applying $\op{enh}$, one finds precisely the universal derivation, because this trivialization of the jet bundles produces a dg Lie algebra over the de Rham complex of $X$, with no complicated higher brackets.
\end{proof}

This identification implies the following corollaries, among many others, due to the work of Arias Abad and Crainic.

\begin{cor}[Proposition 5.1 \cite{AbadCrainic}]
The cohomology of the de Rham complex of $\op{enh}(L)$ is isomorphic to the de Rham cohomology of the underlying manifold~$X$.
\end{cor}

\begin{cor}[Proposition 5.5 \cite{AbadCrainic}]
Let $X$ be a $\fg$-manifold for some finite-dimensional Lie algebra $\fg$. Let $\fg \rtimes X$ denote the associated Lie algebroid: the vector bundle is the trivial $\fg$-bundle on $X$, equipped with the canonical flat connection, and the anchor map is the action map $\rho\co \fg \to \Gamma(X,T_X)$. Then the de Rham complex of $\op{enh}(\fg \rtimes X)$ is the image under $\op{enh}_{\op{mod}}$ of the Kalkman's BRST complex~$W(\fg,X)$.
\end{cor}

\section{Symplectic structures}

The notion of higher symplectic geometry appears in the physics literature as part of the BRST and BV formalisms. Mathematically, there are many approaches, although of most relevance to our results in this section is work of Roytenberg \cite{Roytenberg}, \v{S}evera \cite{Severa}, and \v{S}evera and Weinstein \cite{SW}.
Recently, Pantev, To\"en, Vaqui\'e, and Vezzosi \cite{PTVV} have developed symplectic geometry in the setting of derived algebraic geometry. Here we will provide a bridge between these two mathematical approaches.

In this section we show that the notion of an $n$-shifted symplectic form on $\L8$ spaces is compatible with and extends existing definitions of $n$-symplectic Lie algebroids.  Our explicit comparison is with the formulation of Roytenberg \cite{Roytenberg} in terms of dg manifolds. (We find the lectures of Cattaneo and  Sch\"{a}tz \cite{CatSchat} to be a lucid and efficient exposition of graded and dg manifolds and symplectic structures thereon.)

\subsection{Brief recollections on dg manifolds}

Recall that a {\em graded manifold} is a graded-ringed space $(X,\cA)$ where the underlying topological space $X$ is a smooth manifold and $\cA$ is a sheaf of $\ZZ$-graded commutative algebras locally of the form $U \mapsto \Gamma(U,\Sym E)$, where $E \to U$ is a $\ZZ$-graded vector bundle. A {\em map of graded manifolds} $F = (f,\psi) \co (X,\cA) \to (Y,\cB)$ is a smooth map $f \co X \to Y$ and a map of sheaves of graded $f^{-1}\cinf_Y$-algebras $\psi \co f^{-1} \cB\to \cA$. A {\em vector field} $\fX$ on $(X,\cA)$ is a graded derivation of~$\cA$.

The central example is the graded manifold associated to a graded vector bundle $E \to X$: it is $(X, \cS {\rm ym}_{\cinf}(\cE^\vee))$, where $E^\vee \to X$ denotes the dual vector bundle and $\cE^\vee$ denotes its sheaf of sections. In fact, every graded manifold is isomorphic to a graded manifold coming from a graded vector bundle. We will abusively denote by $E$ the graded manifold arising from the vector bundle~$E~\to~X$.

A {\em dg manifold} is a triple $(X,\cA,Q)$, where $(X,\cA)$ is a graded manifold and $Q$ is a degree 1 vector field on $\cA$ such that $[Q,Q] = 0$. (Such a $Q$ is typically called a ``(co)homological vector field.'') A {\em map of dg manifolds} $F = (f,\psi) \co (X,\cA,Q) \to (Y,\cB,R)$ is a map of graded manifolds such that $\psi$ is cochain map (i.e., $Q\circ\psi=\psi\circ~R$). 

We have already encountered an important class of examples. A Lie algebroid $\rho \co L \to T_X$ produces a dg manifold $(X, \calgd(L))$, where we have compressed notation with $\cA = \calgd(L)^\sharp$ and $Q$ is the differential on the Chevalley-Eilenberg complex of $L$. We denote this dg manifold by $X/L$ as its dg algebra of functions is the derived invariants of $\cinf_X$ with respect to the action of $L$. As Va\u{\i}ntrob \cite{Vaintrob} observed, a Lie algebroid structure on a (ungraded) vector bundle $L \to X$ is equivalent to a dg manifold structure on the graded vector bundle $L[1] \to X$: the cohomological vector field is precisely the differential of the putative $\calgd(L)$ and hence encodes the bracket. 

Another important example of a dg manifold is the {\em de Rham space} $\cX_{dR}$ of a graded manifold $\cX = (X,\cA)$. Suppose, without loss of generality, that $\cX$ corresponds to the graded vector bundle $E \to X$. Then the tangent bundle $T\cX$ corresponds to the graded vector bundle $(T_X \oplus E) \oplus E \to X$. Now, the underlying graded manifold of $\cX_{dR}$ is $T[1]\cX$. As $\cA = \cS {\rm ym}_{\cinf}(\cE^\vee)$, the structure sheaf of $\cX_{dR}$ is 
\[
\cS {\rm ym}_{\cinf}(\cE^\vee \oplus \cT^\vee[-1] \oplus \cE^\vee[-1]) \cong \cS{\rm ym}_{\cA}(\Omega^1_\cA[-1]),
\]
where $\Omega^1_\cA = \Omega^1_X \oplus \cE^\vee$ denotes the sheaf of one-forms for $\cA$. There is a natural degree 1 derivation $d_{dR} \co \cA \to \Omega^1_\cA$ that extends the de Rham differential on $\cinf$; in local homogeneous coordinates $\{x_i\}$, we have the standard formula $d_{dR} = \sum_i dx_i \frac{\partial}{\partial x_i}$. Then $\cX_{dR}$ is the dg manifold $(X, \cS{\rm ym}_\cA(\Omega^1_\cA[-1]), d_{dR})$, where we extend $d_{dR}$ as a derivation to the symmetric algebra.

For $\cX = (X,\cA,Q)$ a dg manifold, it is possible to equip the graded vector space of one-forms $\Omega^1_\cA$ with a natural differential $Q_1$ determined by the requirement that $Q_1 \circ d_{dR} = d_{dR} \circ Q$. This construction amounts to taking the K\"ahler differentials of the dg algebra $(\cA,Q)$, except that we require it to play nicely with smooth functions on $X$, which is not a purely algebraic condition. (In other words, it is the K\"ahler differentials as a dg $\cinf$-ring.) The de Rham space of $\cX$ is then  $(X, \cS{\rm ym}_\cA(\Omega^1_\cA[-1]), d_{dR}+Q_1)$. We will use  this version of the de Rham complex of $(\cA,Q)$, but it seems not to be wholly standard in the dg manifold literature (at least it does not appear in \cite{Roytenberg,CatSchat}).

\begin{prop}
For the dg manifold $X/L$ associated to a Lie algebroid $L$ on $X$, the sheaf of one-forms $\Omega^1_{\calgd(L)}$ is naturally isomorphic to the coadjoint complex of $L$. Moreover, the de Rham complex of $\calgd(L)$ is naturally isomorphic to (the totalization of) the Weil complex of $L$.
\end{prop}

This de Rham complex itself determines a dg manifold $(X, \Omega^*_{X/L})$. 
In light of Corollary \ref{cor:dgversion} and Proposition \ref{prop:weil}, 
we immediately obtain the following corollary.

\begin{cor}\label{cor:derhamofX/L}
For a Lie algebroid $L$ on $X$, the de Rham complex of the dg manifold $X/L$ is mapped by $\op{enh}$ to the de Rham complex of the $\L8$ space~$\op{enh}(L)$.
\end{cor}

\begin{proof}[Proof of proposition]
The underlying graded vector bundles on $X$ are the same: the cotangent bundle for $X/L$ has underlying graded vector bundle $T^\vee_X \oplus L^\vee[-1]$, just like the coadjoint complex. This identification extends to the higher wedge powers. It thus remains to relate the differentials. Since different choices of connection on $L$ induce canonically isomorphic coadjoint complexes, we check locally using a convenient choice. (Compare to the final paragraph of the proof of Proposition \ref{prop:weil}.) Fix coordinates on some small open $U$ in $X$ and fix a frame for $L$ (and the dual frame for $L^\vee$) on this open $U$. Equip $L$ and $L^\vee$ restricted to $U$ with the flat connection associated to this trivialization by the frame. Then the de Rham differential for the dg manifold $U/L$ is identical to the vertical differential of the Weil complex by inspection (see the formulas in Remark 5.2 of \cite{AbadCrainic}, which simplify drastically for a flat connection).
\end{proof}

\begin{remark}
These results are fragments of a larger story. The techniques developed in this paper apply equally well to all dg manifolds, so that a dg manifold provides an $\L8$ space via the Fedosov resolution process we've articulated. The notion of de Rham complex of a dg manifold then goes to the de Rham complex of the associated $\L8$ space.
\end{remark}

\subsection{Shifted symplectic structures}

We will now show that a symplectic structure on the dg manifold associated to a dg Lie algebroid will produce a shifted symplectic structure on its $\L8$ space. Thus standard examples, like Courant algebroids, fit into the $\L8$ space framework.

Let us begin by recalling the definition whose first explicit description is, so far as we know, due to Roytenberg. (See page 6 of \cite{Roytenberg}, just above Lemma 2.2, or Definition 4.3 of \cite{CatSchat}.) This definition is stronger than the definition we consider natural, as we will show momentarily, so we introduce a terminological distinction. 

\begin{definition}
An {\em $n$-shifted Roytenberg symplectic structure} on the dg manifold $X/L = (X, \calgd(L))$ is a 2-form $\omega$ of cohomological degree $n$ on the underlying graded manifold $(X/L)^\sharp := (X, \cS{\rm ym}(\cL^\vee[-1]))$ such that
\begin{enumerate}
\item the induced map $\omega^* \co T_{X/L^\sharp} \to T^\vee_{X/L^\sharp}[n]$ is nondegenerate,
\item $d_{dR} \omega = 0$ in the de Rham space of $(X/L)^\sharp$, and
\item the cohomological vector field $Q = d_{\calgd{L}}$ is symplectic with respect to $\omega$, i.e.,~$L_Q\omega=0$.
\end{enumerate}
\end{definition}

Let us give that definition in other terms. An $n$-shifted Roytenberg symplectic structure on the dg manifold $X/L = (X, \calgd(L))$ is an element $\eta \in \calgd(L,\Lambda^2(\op{adj}\, L)^\vee)$ of cohomological degree $n$ that is closed under the internal differential and also closed under the vertical differential of the Weil complex. (Hence it provides a closed element for the total differential of the Weil complex.) Further, $\eta$ is to be non-degenerate, i.e., $\eta$ induces an isomorphism~$\op{adj}\,L[1]\to(\op{adj}\,L)^\vee [1+n]$.

We consider the following definition more natural, by analogy to our $\L8$ space definition \cite{GGLoop} or the approach of \cite{PTVV}. Let the closed 2-forms $\Omega^{2,cl}(X/L)$ denote the totalization of the double complex
\[
\calgd(L,\Lambda^2((\op{adj}\, L)^\vee[-1])) \xto{d_{dR}} \calgd(L,\Lambda^3 ((\op{adj}\, L)^\vee[-1])) \xto{d_{dR}} \cdots,
\]
and let $i$ denote the obvious truncation to $\calgd(L,\Lambda^2(\op{adj}\, L)^\vee)$. Recall that a {\it closed 2-form} $\omega$ is a cocycle in the complex~$\Omega^{2,cl}(X/L)$. 

\begin{definition}
An {\em $n$-shifted symplectic form} on $X/L$ is a closed 2-form $\omega$ of cohomological degree $n$ such that the induced map $i(\omega) \co \op{adj}\, L[1] \to (\op{adj}\, L)^\vee [1+n]$ is a quasi-isomorphism.
\end{definition}

Explicitly, a closed 2-form $\omega$ of cohomological degree $n$ is really a sequence $\omega = (\omega_0 , \omega_1 , \dotsc )$ with $\omega_i \in \calgd(L,\Lambda^i((\op{adj}\, L)^\vee[-1]))$ of internal degree $n-i+2$ such that
\[
d_{dR} \omega_i = d_H \omega_{i+1}.
\]
Non-degeneracy is only a property of $\omega_0 = i (\omega)$, while being closed is an explicit lift of $\omega_0$ to a cocycle in $\Omega^{2,cl}(X/L)$ and hence involves specifying additional data.

Examples of such shifted symplectic forms come from Roytenberg symplectic structures.

\begin{lemma}
A Roytenberg symplectic structure of degree $n$ defines a $n$-shifted symplectic form on~$X/L$.
\end{lemma}

\begin{proof}
By definition a Roytenberg symplectic structure $\eta$ is concentrated in a single bidegree and is closed under both the horizontal and vertical differential, hence it is closed under the total differential.  Further, by hypothesis $\eta$ is non-degenerate.  Such an $\eta$ corresponds to a $n$-shifted symplectic form of the type $\omega=~(\eta,0,0,\dotsc).$
\end{proof}

By applying Corollary \ref{cor:derhamofX/L}, we deduce the following result from \cite{Roytenberg}.

\begin{cor}
A shifted Roytenberg symplectic structure on $X/L$ produces a shifted symplectic structure on $\op{enh}(L)$. In particular, if $L$ is a Courant algebroid over $X$, then $\op{enh}(L)$ is 2-shifted symplectic.
\end{cor}

\begin{remark}
We expect a Darboux lemma to hold for symplectic $\L8$ spaces, possibly under some further reasonable conditions. (Derived algebraic versions appear in \cite{BBBJ,BouGroj}, where  arguments and conditions can be found that would provide a model for such a lemma.) One consequence would then be that for such spaces, every shifted symplectic structure is locally equivalent to a Roytenberg symplectic structure. 
\end{remark}

\begin{remark}
Subsequent to the posting of this paper, Pym and Safronov \cite{PS} explored shifted symplectic structures on Lie algebroids and their higher generalizations. Their approach to basic definitions (e.g., of the de Rham forms of $X/L$ and shifted symplectic forms) is also modeled on \cite{PTVV}, and it appears to agree with ours where the domains of definition overlap (essentially, dg Lie algebroids concentrated in nonpositive degrees). A key message of \cite{PS} is that Roytenberg's classification must be refined when using these more homotopically sophisticated definitions. For instance, given such a nonpositively-graded dg Lie algebroid, the $\infty$-groupoid of 2-shifted symplectic structures is equivalent to a 2-groupoid of twisted Courant algebroids. That means that a 2-shifted symplectic structure can be specified by a 2-form that is {\em strictly} closed and nondegenerate but that also involves a closed 4-form  on the underlying smooth manifold (namely, the ``twist'' of the Courant algebroid). Hence a Safronov-Pym 2-shifted symplectic structure is more general than a Roytenberg 2-symplectic structure. (See Section 5 for a precise discussion.) We expect---but do not verify here---that their results port to our context, in which case one can deduce that there are shifted symplectic structures on $\op{enh}(L)$ that do not arise from shifted Roytenberg symplectic structures. In other words, the converse to the preceding corollary is not true. 
\end{remark}

\subsection{AKSZ Theories}

Shifted symplectic structures play a central role in the AKSZ construction \cite{AKSZ}, 
a mechanism for producing classical field theories in the BV formalism.  
These theories are $\sigma$-models, i.e., the field content consists of maps between geometric entities. The paradigm operates as follows:
\begin{enumerate}
\item The first input is a {\it source} dg manifold $\Sigma$ equipped with a $d$-orientation, which produces a volume form $\mathsf{dvol}$ of degree $d$, e.g., a closed oriented $d$-manifold or a Calabi-Yau $d$-fold.
\item The other input is a {\it target} dg manifold $X$, which is equipped with a $k$-shifted symplectic structure~$\omega$.
\item The {\it space of fields}, or {\it field content}, is the mapping space $\mathrm{Map} (\Sigma, X)$, which we denote by~$\sE$.
\item The form $\mathsf{dvol}$ and the symplectic structure $\omega$ induce a $k-d$ symplectic structure on the space of fields $\sE = \mathrm{Map} (\Sigma, X)$. Explicitly, for $\varphi \co \Sigma \to X$ a map, the tangent space $T_\varphi \sE$ has a pairing of degree $k-d$ given by
\[
\langle \xi_1 , \xi_2 \rangle := \int_{p \in \Sigma} \varphi^\ast \omega (p) \left ( \xi_1 (p) , \xi_2 (p) \right ) \mathsf{dvol},  
\]
where  $\xi_1 , \xi_2 \in T_\varphi \sE = \Gamma ( \Sigma, \varphi^\ast T_X)$.
\item In the case where $k-d = -1$, the fields $\sE$ form a $-1$-symplectic space and hence defines a classical BV theory. 
 \end{enumerate}
This paradigm was used by \cite{AKSZ} to interpret several important field theories, including Chern-Simons theory and the $A$- and $B$-models of mirror symmetry. See also the more recent work of Cattaneo, Mn\"{e}v, and collaborators, e.g., Section 2 of~\cite{CM}.

This methodology can be extended beyond the setting of dg manifolds: 
in \cite{PTVV} this formalism is developed for shifted symplectic derived stacks. 
As noted in \cite{GGLoop}, the AKSZ formalism extends to shifted symplectic $L_\infty$ spaces, where it provides perturbative descriptions of these $\sigma$-models in a way compatible with the renormalization/BV package developed by Costello. 
Moreover, these perturbative theories are presented as a families of gauge theories parametrized by the target manifold. 
This methodology has been fruitfully exploited in the following recent works:
\begin{itemize}
\item The formal neighborhood of the zero section $X \hookrightarrow T^\ast X$ determines a 0-shifted symplectic $L_\infty$ space. Hence, there is an associated one-dimensional theory that is a version of (topological) quantum mechanics. This theory is quantized in \cite{GGCS} and the observable theory is described in \cite{GGW}. 
\item Any symplectic manifold $(M, \omega)$ defines a 0-shifted symplectic $L_\infty$ space. The quantization of the resulting one-dimensional theory \cite{GLL} recovers Fedosov quantization \cite{Fedosov}  and gives a new proof of the algebraic index theorem~\cite{NT}.
\item The formal neighborhood of the zero section $Y \hookrightarrow T^\ast Y$ of a complex manifold $Y$ determines a 0-shifted symplectic $L_\infty$ space. If one takes a Calabi-Yau Riemann surface as the source, one obtains an AKSZ theory known as the curved $\beta\gamma$ system. Its quantization recovers the sheaf of chiral differential operators \cite{CDO} on $Y$ and its partition function gives a mathematically rigorous interpretation of the Witten genus in terms of QFT~\cite{CosWG}. 
\item The three-dimensional theory known as Rozansky-Witten theory is realized via a 2-shifted symplectic $L_\infty$ space arising from a holomorphic symplectic manifold. See  \cite{CCLL} for the details of the $L_\infty$ space formulation and its quantization. 
\item Perturbative aspects of the Riemannian $\sigma$-model in two dimensions also fit into this paradigm. In \cite{GW} this theory is described in terms of a 1-shifted symplectic $L_\infty$ space. The $\beta$-function (at one loop) of this theory describes Ricci flow on the target manifold, as  first explained by Friedan~\cite{Friedan}. 
\end{itemize}
 
The results of {\em this} paper allow one to describe a whole slew of additional perturbative BV theories. 
For example, a Courant algebroid $(E, \rho, \langle , \rangle)$ defines a 2-shifted symplectic $L_\infty$ space.  
The quantization of the resulting three-dimensional AKSZ theory describes a low-energy effective theory for the {\it Courant $\sigma$-model} \cite{Ikeda}.
In collaboration with Brian Williams, the authors intend to quantize this theory and explore its implications. Indeed, this example nearly brings the story full circle, as much of Roytenberg's work \cite{Roy1} was motivated by giving a mathematically coherent description to classical aspects of the Courant $\sigma$-model.

\appendix

\section{Jets and connections}\label{app}

If $E \to X$ is a vector bundle on a smooth manifold, its sheaf $\cE$ of smooth sections is a module over the sheaf $\cinf_X$ of smooth functions on $X$. We will be somewhat cavalier in moving back and forth between a vector bundle $E$ and its sheaf $\cE$ of smooth sections; in general, we will use Roman script for vector bundles and calligraphic script for sheaves. We will focus on the category $\VB(X)$ whose objects are finite-rank vector bundles and whose morphisms are vector bundle maps. It has a natural symmetric monoidal structure given by the Whitney tensor product, given by fiberwise tensor product.

To simplify notation, we will suppress subscripts referring to the base manifold $X$, so $\cinf$ will mean $\cinf_X$ and $T^\vee$ will mean $T^\vee_X$, for example.

\subsection{Finite jet bundles}

For every natural number $k$, there is a bundle $J^k(E) \to X$ of {\it $k$-jets of $E$}, whose fiber at a point $x \in X$ is $\cE_x/\fm_x^{k+1}$, where $\cE_x$ denotes the stalk of $\cE$ at $x$ (i.e., the vector space of germs at $x$ of sections) and where $\fm_x$ denotes the vector space of germs at $x$ of smooth functions vanishing at $x$. Observe that~$J^0(E)=E$. 

%There is a short exact sequence of vector bundles
%\[
%\Sym^k(T^*) \ot E \to J^k(E) \to J^{k-1}(E),
%\]
%arising from the natural map forgetting a $k$-jet to a $k-1$-jet. There is thus a non-canonical isomorphism $J^k(E) \cong \Sym^{\leq k}(T^*) \ot E$.

Let $\cJ^k(E)$ denote the sheaf of smooth sections of $J^k(E)$. There is a map of sheaves $j_k \co \cE \to \cJ^k(E)$ sending a smooth section to its $k$-jet. (Note that $j_k$ does not arise from a map of vector bundles for $k > 0$.) Two smooth sections $s,s'$ ``agree to order $k$ at $x$'' if their germs $j_k(s), j_k(s')$ are the same in~$\cJ^k(E)_x$.

Jets play an important role in relation to differential operators. For instance, an order $k$ differential operator $P \co \cE \to \cF$ is a $\CC$-linear map of sheaves that factors as $\cE \xto{j_k} \cJ^k(E) \xto{\tilde{P}} \cF$, where $\tilde{P}$ is a $\cinf$-linear map of sheaves.

\subsection{The $\infty$-jet bundles}

The {\it $\infty$-jet bundle} $J(E)$ is the pro-finite-rank vector bundle $\lim_k J^k(E)$. We will work with it as an infinite-rank vector bundle with filtration $F^k J(E)$ whose quotients $J(E)/F^{k+1} J(E) \cong J^k(E)$ are finite-rank. All our constructions will respect this filtration; in other words, we work in the category of filtered vector bundles and require maps to be filtration-preserving. 

We use $\cJ(E)$ to denote the sheaf of smooth sections of $J(E)$. There is a sheaf map $j_\infty \co \cE \to \cJ(E)$ sending a smooth section to its $\infty$-jet. In local coordinates around a point $x$ and with a choice of trivialization of $E$ around $x$, this map agrees with the Taylor expansion. In other words, one should view these jet bundles as a coordinate-free way of working with Taylor expansions. 

Let $\cJ$ denote the sheaf of $\infty$-jets of smooth functions, i.e., $\cJ = \cJ(\CC)$ for the trivial rank-one bundle $\CC$. Similarly, let $\cJ^k$ denote the sheaf of $k$-jets of smooth functions. 

\subsubsection{}

The relationship of jets with differential operators provides a useful alternative construction for jets. Recall from above that $\cD_{\leq k} = \shom_\cinf(\cJ^k, \cinf)$ is the sheaf of order $k$ differential operators on $\cinf$. Let $\cD = \colim_k \cD_{\leq k}$ denote the sheaf of differential operators on $\cinf$. We see that $\cJ \cong \shom_{\cinf}(\cD,\cinf)$, where the ascending filtration on $\cD$ induces the descending filtration on~$\cJ$.

\begin{lemma}
The sheaf $\cJ$ is a sheaf of commutative algebras over~$\cinf$.
\end{lemma}

\begin{proof}
Recall that $\cD$ has a cocommutative coproduct $\kappa \co \cD \to \cD \otimes_{\cinf} \cD$. In brief, for every pair of left $\cD$-modules $\cM$ and $\cN$, there is a natural left $\cD$-module structure on $\cM \ot_\cinf \cN$, where
\[
X \cdot (m \ot n) = (X \cdot m) \ot n + m \ot (X \cdot n)
\]
for any vector field $X$ and any section $m$ of $\cM$ and $n$ of $\cN$. Now set $\cM = \cD = \cN$  and construct the map of left $\cD$-modules via $\kappa \co 1 \mapsto 1 \ot 1$. Explicit computation verifies $\kappa$ is a cocommutative coproduct.

There is thus a natural map
\[
\kappa^* \co \shom_{\cinf}(\cD \otimes_{\cinf} \cD,\cinf) \to \shom_{\cinf}(\cD, \cinf) = \cJ,
\]
and precomposing with the natural map
\[
\begin{array}{cccc}
\cJ \otimes_{\cinf} \cJ =& \shom_{\cinf}(\cD, \cinf) \otimes_{\cinf} \shom_{\cinf}(\cD,\cinf) & \to & \shom_{\cinf}(\cD \otimes_{\cinf} \cD, \cinf) \\
& \phi \otimes \psi & \mapsto & (P \otimes Q \mapsto \phi(P) \psi(Q)),
\end{array}
\]
we obtain a canonical map $m \co \cJ \otimes_{\cinf} \cJ \to \cJ$, which is indeed a commutative product.
\end{proof}

It will be convenient below to see this commutative algebra structure in local coordinates. Fix coordinates $\{x_1,\ldots,x_n\}$ in a small neighborhood $U$ around some point $p \in X$. Then we obtain a natural associated frame $\{\partial/\partial x_i\}$ for the tangent bundle on $U$, and hence every order $k$ differential operator has a unique expression 
\[
P = \sum_{\vec{m} \in \NN^n} f_{\vec{m}}(x) \frac{\partial^{\vec{m}}}{\partial x^{\vec{m}}} \quad\text{where}\quad\vec{m} = (m_1,\ldots,m_n) \text{ and } \sum_{i=1}^n m_i \leq k.
\]
In other words, such an operator looks like a degree $k$ polynomial in the partial derivatives, with coefficients in $\cinf(U)$. Locally, we thus see that there is an isomorphism $\cD(U) \cong {\mathcal{S}}{\rm ym}_{\cinf(U)}(\cT(U))$ as $\cinf_U$-modules (but not as algebras). As jets are the fiberwise linear dual, we see that $\cJ \cong \calsym_{\cinf}(\cT^\vee)$ as sheaves on $U$. Sections are formal power series in the dual frame $\{dx_i\}$ for the cotangent bundle. Indeed, there is a natural trivialization of the jet bundle $\cJ$ on~$U$
\[
\cJ(U) \cong \cinf(U) \otimes_\CC \CC[[x_1,\ldots,x_n]]
\]
so that the $\infty$-jet of a smooth function $\phi$ is given by
\[
\sum_{\vec{m} \in \NN^n} (\partial_{\vec{m}} \phi) x^{\vec{m}}
\]
with $\vec{m} = (m_1, \ldots, m_n)$ a multi-index. (In other words, we are just giving the pointwise Taylor expansion of $\phi$.) Multiplication in $\cJ$ locally is just multiplication of these power series.

\subsubsection{} One property will be crucial in our work. It is undoubtedly well-known to experts but does not seem to be in the literature, so we provide a proof.

Let $\Mod_\cJ^{fil}$ denote the category of filtered $\cJ$-module sheaves (i.e., possessing a filtration compatible with that on $\cJ$) and with maps the filtration-preserving $\cJ$-module maps.

\begin{prop}\label{prop:jetsm}
The $\infty$-jet functor $\cJ(-) \co \VB(X) \to \Mod_{\cinf}$ described above, sending $E$ to $\cJ(E)$, lifts to a symmetric monoidal functor $\cJ(-) \co (\VB(X),\ot)\to(\Mod^{fil}_\cJ, \ot_\cJ)$.
\end{prop}

In other words, for every vector bundle $E$, the sheaf $\cJ(E)$ has a canonical $\cJ$-module structure. Moreover, there is a canonical isomorphism $j_{E,F} \co \cJ(E) \otimes_\cJ \cJ(F) \to \cJ(E \otimes F)$ for every pair of vector bundles $E$ and~$F$.

\begin{proof}
We will first construct a natural morphism $j_{E,F}$ and then verify in local coordinates that it is an isomorphism. Using the coproduct map $\kappa \co \cD \otimes_{\cinf} \cD \to \cD$ from the preceding proof, there is a composition of maps
\[
\begin{array}{ccccc}
\shom_{\cinf}(\cD,\cE) \otimes_{\cinf} \shom_{\cinf}(\cD,\cF) & \to & \shom_{\cinf}(\cD \otimes_{\cinf} \cD, \cE \otimes \cF) &\xto{\kappa^*} & \shom_{\cinf}(\cD, \cE \otimes \cF), \\
\shortparallel & & & & \shortparallel \\
\cJ(E) \otimes_{\cinf} \cJ(F) & & & & \cJ(E \otimes F)
\end{array}
\]
which provides~$j_{E,F}$.

When $E$ is the trivial bundle, this map equips $\cJ(F)$ with a $\cJ$-module structure. The natural filtration on $\cJ(F)$ is automatically compatible with the filtration on~$\cJ$. 

Indeed, sufficiently locally, $F$ becomes a trivial bundle of rank $r$ via a frame $\{f_1,\ldots,f_r\}$. A choice of coordinates trivializes $\cJ \cong \cinf(U) \otimes_\CC \CC[[x_1,\ldots,x_n]]$, as described above. Then locally 
\[
\cJ(F)(U) \cong \cinf(U) \otimes_\CC \CC[[x_1,\ldots,x_n]] \otimes \bigoplus_{j=1}^r \CC f_j,
\]
and the action of $\cJ(U)$ on $\cJ(F)(U)$ is simply the natural multiplication by power series with coefficients in~$\cinf(U)$.

For arbitrary $E$ and $F$, we want to show that this map coequalizes the action of $\cJ$ on the $\cJ(E)$ and $\cJ(F)$. In other words, we want it to factor through the map $\cJ(E) \otimes_{\cinf} \cJ(F) \to\cJ(E)\otimes_\cJ~\cJ(F)$.

Now that we have an explicit map, we can straightforwardly check that it is an isomorphism. Fix a point $p \in M$ and choose coordinates $\{x_1,\ldots,x_n\}$ in a neighborhood $U$ of $p$. As noted above, there is then a natural trivialization of the jet bundle $\cJ$ on~$U$
\[
\cJ(U) \cong \cinf(U) \otimes_\CC \CC[[x_1,\ldots,x_n]]
\]
Fix a frame $\{e_i\}$ for $E$ on $U$ and a frame $\{f_j\}$ for $F$ on $U$. Then
\[
\cJ(E)(U) \cong \cinf(U) \otimes_\CC \CC[[x_1,\ldots,x_n]] \otimes \bigoplus_i \CC e_i
\]
and
\[
\cJ(F)(U) \cong \cinf(U) \otimes_\CC \CC[[x_1,\ldots,x_n]] \otimes \bigoplus_j \CC f_j.
\]
Both $\cJ(E)(U)$ and $\cJ(F)(U)$ are free modules over $\cinf(U)[[x_1,\ldots,x_n]]$.
Likewise,
\[
\cJ(E \otimes F)(U) \cong \cinf(U) \otimes_\CC \CC[[x_1,\ldots,x_n]] \otimes \bigoplus_{i,j} \CC e_i \otimes f_j
\]
The map $j_{E,F}(U)$ reduces to the map 
\[
(\phi \otimes e_i) \otimes (\psi \otimes f_j) \mapsto (\phi \psi) \otimes (e_i \otimes f_j),
\]
where $\phi, \psi \in \cinf(U)[[x_1,\ldots,x_n]]$. Hence, it coequalizes the action of~$\cJ$.
\end{proof}

\begin{remark}
\label{rmk: jet of diff op}
Note that a differential operator $P \co \cE \to \cF$,  maps to an operator $J(P) \co \cJ(E) \to \cJ(F)$  that is only $\cinf$-linear, not $\cJ$-linear. Therefore, there is an extension of the $\infty$-jet functor to a functor from the category of differential complexes but it lands in the category $\cD$-modules, not~$\Mod_\cJ^{fil}$.
\end{remark}

\subsection{Three important constructions}

\subsubsection{Splittings and Fedosov resolutions}\label{sect:splitting}

The following non-canonical descriptions of jet bundles play an important role in reinterpreting Lie algebroid constructions as Lie algebraic constructions over the de Rham complex.

There is a slick way of understanding the sheaves $\cJ$ and $\cJ^k$. Consider the diagonal embedding $\Delta \co X \to X \times X$ and pull back the sheaf $\cinf_{X \times X}$ along $\Delta$. There is a canonical quotient map $q \co \Delta^{-1} \cinf_{X \times X} \to \cinf_X$ given by restricting a function to the diagonal, and let $\cI_\Delta$ denote the kernel---the functions that vanish on the diagonal---which is a sheaf of ideals inside the sheaf $\Delta^{-1} \cinf_{X \times X}$ of algebras. Then, by our definition, $\cJ^k = \Delta^{-1} \cinf_{X \times X}/\cI_\Delta^{k+1}$ for every $k$, and~$\cJ=\lim\Delta^{-1}\cinf_{X \times X}/\cI_\Delta^{k+1}$.

\begin{lemma}
There exists a non-canonical isomorphism $\sigma \co \calsym(T^\vee) \to \cJ$ of filtered algebras over~$\cinf$. 
\end{lemma}

\begin{proof}
Fix a connection $\nabla$ on the tangent bundle $T_X \to X$. The associated exponential map $\exp_{\nabla}$ produces a diffeomorphism between a tubular neighborhood of the zero section $X \xto{\op{zero}} TX$ and a tubular neighborhood of the diagonal $X \xto{\Delta} X \times X$. Thus, on that tubular  neighborhood, we obtain an isomorphism of short exact sequences of sheaves
\[
\xymatrix{
\cI_\Delta \ar@{^{(}->}[r] \ar[d]^{\exp_\nabla^{-1}} & \Delta^{-1} \cinf_{X \times X} \ar[r] \ar[d]^{\exp_\nabla^{-1}} & \cinf_X \ar[d]^{\id}\\
\cI_{\op{zero}} \ar@{^{(}->}[r] & \op{zero}^{-1} \cinf_{TX} \ar[r] & \cinf_X
}
\]
%\[
%\begin{tikzcd}
%\cI_\Delta \arrow[r,hook] \arrow[d,"\exp_\nabla^{-1}"] & \Delta^{-1} \cinf_{X \times X} \arrow[r] \arrow[d,"\exp_\nabla^{-1}"] & \cinf_X \arrow[d,"\id"]\\
%\cI_{\op{zero}} \arrow[r,hook] & \op{zero}^{-1} \cinf_{TX} \arrow[r] & \cinf_X
%\end{tikzcd}
%\]
where $\cI_{\op{zero}}$ denotes the ideal sheaf encoding functions on $T_X$ vanishing on the zero section. We thus obtain isomorphisms 
\[
\Delta^{-1} \cinf_{X \times X}/ \cI_{\Delta}^{k+1} = \cJ^k \xto{\cong} \cS{\rm ym}^{\leq k}(T^\vee) = \op{zero}^{-1} \cinf_{T_X}/\cI_{\op{zero}}^{k+1}
\]  
for all $k$. Taking the limit, we obtain the claim.
\end{proof}

There is a natural way to extend this type of construction to vector bundles over~$X$.

\begin{lemma}
Fix a connection $\nabla_{T_X}$ on the tangent bundle $T_X \to X$. A connection $\nabla_{E}$ on a vector bundle $E \to X$ then induces a splitting $\sigma_\nabla~\co~E~\to~J(E)$.
\end{lemma}

\begin{proof}
For a point $x \in X$, use the exponential map for the connection $\nabla_{T_X}$ to parametrize a small neighborhood of $x$ in $X$. For each point $e$ in the fiber of $E$ over $x$, use the connection $\nabla_E$ to extend to a smooth section $s_e$ of $E$ over that small neighborhood of $x$. Then $j_\infty(s_e)$ is a section of $J(E)$ on that small neighborhood. Set $\sigma_\nabla(e)$ to be the value of $j_\infty(s_e)$ at $x$. This construction in fact produces a vector bundle map.
\end{proof}

\begin{cor}
For every vector bundle $E$, there is a non-canonical isomorphism $m_\sigma \co \cJ \otimes_{\cinf} \cE \to \cJ(E)$ as $\cJ$-modules.
\end{cor}

\begin{proof}
Fix a splitting $\sigma$ of the canonical quotient $\cJ(E) \to \cE$. Then we obtain a natural map $m_\sigma \co \cJ \otimes_{\cinf} \cE \to \cJ(E)$ sending $\phi \otimes e$ to $\phi \cdot \sigma(e)$, using the $\cJ$-module structure of $\cJ(E)$. It is then straightforward to check locally that $m_\sigma$ is an isomorphism. (See, e.g., the proof of Lemma E.2 of~\cite{GGCS}.)
\end{proof}

\begin{cor}
Fix an algebra isomorphism $\sigma_0 \co \calsym(T^\vee_X) \to \cJ$. Let $E \to X$ be a vector bundle. There exists a non-canonical isomorphism $\sigma_1 \co \calsym(T^\vee_X) \ot \cE \to \cJ(E)$ intertwining the module structures.
\end{cor}

\begin{proof}
Fix a splitting $\sigma$ of the canonical quotient $\cJ(E) \to \cE$, as in the proof of the proceeding lemma. Then set $\sigma_1 = m_\sigma\circ\sigma_0\otimes\id_\cE$.
\end{proof}

\subsubsection{The Grothendieck connection}

As notation, we mention that for $V \to X$ a vector bundle with flat connection $\nabla$, we use $dR(V,\nabla)$ to denote the sheaf of cochain complexes
\[
\cV \xto{\nabla} \Omega^1(V) \xto{\nabla} \Omega^2(V) \xto{\nabla} \cdots
\]
called the {\it de Rham complex of~$(V,\nabla)$}.

Note that there is a composite functor $dR(J(-))$ from differential complexes on $X$ to $\Omega_X$-modules. Further, note that any $\Omega_X$-module inherits a filtration via the nilpotent ideal~$\Omega^{\ge 1}_X$.

\begin{prop}\label{prop:grconn}
Let $E \to X$ be a vector bundle. The bundle $J(E) \to X$ has a canonical flat connection $\nabla_E$, called the Grothendieck connection, such that 
\[
j_\infty \co \cE \to dR(J(E), \nabla_E)
\]
is a quasi-isomorphism of sheaves. In particular, the (hyper)cohomology of $dR(J(E),\nabla_E)$ vanishes except in degree~0.
\end{prop}

\begin{proof}
Since $\cE$ is soft, it has vanishing higher cohomology, so the second claim follows from the first. It now suffices to demonstrate this quasi-isomorphism locally. By picking a frame on $E$ in some small coordinatized neighborhood on which $E$ trivializes, one is left with verifying the claim for the trivial rank $r$ bundle. See Appendix B of \cite{CFT} for an explicit contracting homotopy. 
\end{proof}

\subsubsection{Jets and pullbacks}

We will need to understand how the jet construction intertwines with maps of manifolds. It plays a crucial role Proposition \ref{fullfunctoriality}, which explains maps betwen Lie algebroids living over different manifolds provide maps of the associated $\L8$ spaces. The following proposition is undoubtedly known by experts, but we could not find a convenient reference, so we provide a proof.

\begin{prop}\label{prop:pulljet}
Let $f \co X \to Y$ be a map of smooth manifolds and  $E \to Y$ a vector bundle.  Then there is a natural map of complete filtered (i.e., pro-) vector bundles on~$X$:
\[
f^{-1}(J_Y E) \to J_X(f^{-1} E).
\]
\end{prop}

\begin{proof}
The key is to use the geometric perspective on the jet construction as described at the beginning of Section \ref{sect:splitting} above.

Let $f^2: X^2 \to Y^2$ denote the map $(x,x') \mapsto (f(x),f(x'))$. We thus have a commuting diagram 
\[
\xymatrix{
X \ar[r]^{f} \ar[d]^{\Delta_X} & Y \ar[d]^{\Delta_Y} \\
X^2 \ar[r]^{f^2} & Y^2
}
\]
%\[
%\begin{tikzcd}
%X \arrow[r,"f"] \arrow[d, "\Delta_X"] & Y \arrow[d, "\Delta_Y"] \\
%X^2 \arrow[r,"f^2"] & Y^2
%\end{tikzcd}
%\]
of manifolds. Hence, the short exact sequence of $\cinf_{Y}$-module sheaves
\[
0 \to \cI_{\Delta_Y} \to \Delta^{-1}_Y \cinf_{Y^2} \to \cinf_Y \to 0
\]
pulls back to a map of right exact sequences of sheaves of vector spaces
\[
\xymatrix{
f^{-1} \cI_{\Delta_Y} \ar[r] \ar[d] & f^{-1}\Delta^{-1}_Y \cinf_{Y^2} \ar[r] \ar[d] &  f^{-1}\cinf_Y \ar[r] \ar[d] & 0\\
\cI_{\Delta_X} \ar[r]& \Delta^{-1}_X \cinf_{X^2} \ar[r] & \cinf_X \ar[r] & 0
}
\]
%\[
%\begin{tikzcd}
%f^{-1} \cI_{\Delta_Y} \arrow[r] \arrow[d] & f^{-1}\Delta^{-1}_Y \cinf_{Y^2} \arrow[r] \arrow[d] &  f^{-1}\cinf_Y \arrow[r] \arrow[d] & 0\\
%\cI_{\Delta_X} \arrow[r]& \Delta^{-1}_X \cinf_{X^2} \arrow[r] & \cinf_X \arrow[r] & 0
%\end{tikzcd}
%\]
by the functoriality of pull back~$f^{-1}$.

The middle vertical arrow is a map of sheaves of commutative algebras and the left vertical arrow is a map of sheaves of ideals. Thus we can look at quotients by powers of the ideal sheaves, and we obtain a canonical map
\[
f^{-1} \cJ^k_Y \to \cJ^k_X
\]
for every $k$. These maps induce a map of the associated pro-vector bundles. Concretely, this map describes how the Taylor expansion of a function $\phi$ around a point $f(x)$ in $Y$ relates to the Taylor expansion of $f^{-1} \phi = \phi \circ f$ around~$x$.

We can do something similar with a vector bundle $E \to Y$. Let $\pi_1^{-1}E \to Y^2$ denote the vector bundle pulled back along the projection $\pi_1 \co Y^2 \to Y$ to the first copy of $Y$. Let $\cE$ denote the sheaf of smooth sections of $E$ on $Y$, and let $\cE^{(2)}$ denote the sheaf of smooth functions of $\pi_1^E$ on $Y^2$. The short exact sequence of $\cinf_{Y}$-module sheaves
\[
0 \to \cI_{\Delta_Y} \otimes_{\cinf_Y} \Delta^{-1}_Y \cE^{(2)} \to \Delta^{-1}_Y \cE^{(2)} \to \cE \to 0
\]
pulls back to a map of right exact sequences of sheaves of vector spaces
\[
\xymatrix{
f^{-1} (\cI_{\Delta_Y} \otimes_{\cinf_Y} \Delta^{-1}_Y) \cE^{(2)} \ar[r] \ar[d] & f^{-1}\Delta^{-1}_Y  \cE^{(2)} \ar[r] \ar[d] &  f^{-1}\cE \ar[r] \ar[d] & 0\\
\cI_{\Delta_X} \otimes_{\cinf_X} \cE_f^{(2)} \ar[r]& \Delta^{-1}_X \cE_f^{(2)} \ar[r] & \cE_f \ar[r] & 0
}\]
%\[
%\begin{tikzcd}
%f^{-1} (\cI_{\Delta_Y} \otimes_{\cinf_Y} \Delta^{-1}_Y) \cE^{(2)} \arrow[r] \arrow[d] & f^{-1}\Delta^{-1}_Y  \cE^{(2)} \arrow[r] \arrow[d] &  f^{-1}\cE \arrow[r] \arrow[d] & 0\\
%\cI_{\Delta_X} \otimes_{\cinf_X} \cE_f^{(2)} \arrow[r]& \Delta^{-1}_X \cE_f^{(2)} \arrow[r] & \cE_f \arrow[r] & 0
%\end{tikzcd}
%\]
where $\cE_f$ denotes the sheaf of smooth sections of $f^{-1} E$ on $X$ and $\cE_f^{(2)}$ denotes the sheaf of smooth sections of $\pi_1^{-1} f^{-1} E$ on $X^2$. Thus by looking at quotients by powers of the ideal sheaves, we obtain a canonical map
\[
f^{-1} \cJ^k_Y(E) \to \cJ^k_X(f^{-1}E).
\]
In the limit, we obtain a natural map $f^{-1} \cJ_Y(E) \to \cJ_X(f^{-1}E)$.
\end{proof}

\bibliographystyle{alpha}
\bibliography{algebroids_v12}

\end{document}